%% file: main.tex
\documentclass[11pt,fceqn]{article}

\usepackage{fullpage}
\usepackage[numbers,sort&compress]{natbib}
\usepackage{algorithm}
\usepackage{algpseudocode}
\usepackage{amsmath,amsthm,amsfonts}
\usepackage{amsmath}
\usepackage{color}
\usepackage{mathrsfs}
\usepackage{enumitem}
\usepackage{bm}
\usepackage{multirow}
\usepackage{booktabs}
\usepackage{makecell}
\usepackage{graphicx}
\usepackage{subfigure}
\usepackage{caption}
\usepackage{comment}
\usepackage{cases}
\usepackage{epsfig}

\usepackage[colorlinks, linkcolor=black, citecolor=red, urlcolor=red]{hyperref}

\include{notations}

\numberwithin{equation}{section}

\theoremstyle{definition}
\newtheorem{thm}{Theorem}[section]
\newtheorem{lem}[thm]{Lemma}

\newtheorem{rem}[thm]{Remark}

\newtheorem{assumption}{Assumption}

\begin{document}

\title{\textbf{A Conservation Law Method in Optimization}}

\author{Bin Shi}

\maketitle
\begin{abstract}
\input{abstract}

\end{abstract}

\input{01_introduction}

\input{02_related}

\input{03_symplectic}

\input{04_local}

\input{05_experimental}
\input{06_concl}

\bibliographystyle{plainnat}
\bibliography{sigproc}

\newpage
\appendix

\end{document}

%% file: abstract.tex
We propose some algorithms to find local minima in nonconvex optimization and to obtain global minima in some degree from the Newton Second Law without friction. With the key observation of the velocity observable and controllable in the motion, the algorithms simulate the Newton Second Law without friction based on symplectic Euler scheme. From the intuitive analysis of analytical solution, we give a theoretical analysis for the high-speed convergence in the algorithm proposed. Finally, we propose the experiments for strongly convex function, non-strongly convex function and nonconvex function in high-dimension.

%% file: 01_introduction.tex
\section{Introduction}
\label{sec:introduction}

Non-convex optimization is the dominating algorithmic technique behind many state-of-art results in machine learning, computer vision, natural language processing and reinforcement learning. Finding a global minimizer of a non-convex optimization problem is NP-hard. Instead, the local search method become increasingly important, which is based on the method from convex optimization problem. Formally, the problem of unconstrained  optimization is stated in general terms as that of finding the minimum value that a function attains over Euclidean space, i.e.
$$
\mathop{\text{min}}_{x \in \mathbb{R}^{n}}f(x).
$$
Numerous methods and algorithms have been proposed to solve the minimization problem, notably gradient methods, Newton's methods, trust-region method, ellipsoid method  and interior-point method~\citep{polyak1987introduction,nesterov2013introductory,wright1999numerical,luenberger1984linear,boyd2004convex,bubeck2015convex}.

First-order optimization algorithms are the most popular algorithms to perform optimization and by far the most common way to optimize neural networks, since the second-order information obtained is supremely expensive.  The simplest and earliest method for minimizing a convex function $f$ is the gradient method, i.e.,
\begin{equation}
\label{eqn:gradient}
\left\{\begin{aligned}
& x_{k+1} = x_{k} - h \nabla f(x_{k})\\
& \text{Any Initial Point}: \; x_{0}.
\end{aligned}\right.
\end{equation}
There are two significant improvements of the gradient method to speed up the convergence. One is the momentum method, named as Polyak heavy ball method, first proposed in~\citep{polyak1964some}, i.e.,
\begin{equation}
\label{eqn:momentum}
\left\{\begin{aligned}
& x_{k+1} = x_{k} - h \nabla f(x_{k}) + \gamma_{k} (x_{k} - x_{k-1})\\
& \text{Any Initial Point}: \; x_{0}.
\end{aligned}\right.
\end{equation}
Let $\kappa$ be the condition number, which is the ratio of the smallest eigenvalue and the largest eigenvalue of Hessian at local minima. The momentum method speed up the local convergence rate from $1 - 2\kappa$ to $1 - 2\sqrt{\kappa}$. The other is the notorious Nesterov's accelerated gradient method, first proposed in~\citep{nesterov1983method} and an improved version~\citep{nesterov1988general, nesterov2013introductory}, i.e.
\begin{equation}
\label{eqn:nesterov}
\left\{\begin{aligned}
& y_{k+1} = x_{k} - \frac{1}{L} \nabla f(x_{k}) \\
& x_{k+1} = x_{k} + \gamma_{k}(x_{k+1} - x_{k})\\
& \text{Any Initial Point}: \; x_{0} = y_{0}
\end{aligned}\right.
\end{equation}
where the parameter is set as 
$$
\gamma_{k} = \frac{\alpha_{k} (1 - \alpha_{k})}{\alpha_{k}^{2} + \alpha_{k+1}}\quad \text{and}\quad \alpha_{k+1}^{2} = (1 - \alpha_{k+1})\alpha_{k}^{2} + \alpha_{k+1} \kappa.
$$
The scheme devised by Nesterov does not only own the property of the local convergence for strongly convex function, but also is the global convergence scheme, from $1 - 2\kappa$ to $1 - \sqrt{\kappa}$ for strongly convex function and from $\mathcal{O}\left(\frac{1}{n}\right)$ to $\mathcal{O}\left(\frac{1}{n^{2}}\right)$ for non-strongly convex function.

Although there is the complex algebraic trick in Nesterov's accelerated gradient method, the three methods above can be considered from continuous-time limits~\citep{polyak1964some, su2014differential, wibisono2016variational, wilson2016lyapunov} to obtain physical intuition. In other words, the three methods can be regarded as the discrete scheme for solving the ODE. The gradient method~(\ref{eqn:gradient}) is correspondent to 
\begin{equation}
\label{eqn:gradient_system}
\left\{\begin{aligned}
& \dot{x} = - \nabla f(x_{k})\\
& x(0)= x_{0},
\end{aligned}\right.
\end{equation}
and the momentum method and Nesterov accelerated gradient method are correspondent to 
\begin{equation}
\label{eqn:2nd_system}
\left\{\begin{aligned}
& \ddot{x} + \gamma_{t} \dot{x} + \nabla f(x) = 0\\
& x(0)= x_{0}, \;\dot{x}(0) =0,
\end{aligned}\right.
\end{equation}
the difference of which are the setting of the friction parameter $\gamma_{t}$. There are two significant intuitive physical meaning in the two ODEs~(\ref{eqn:gradient_system}) and~(\ref{eqn:2nd_system}). The ODE~(\ref{eqn:gradient_system}) is the governing equation for potential flow, a correspondent phenomena of waterfall from the height along the gradient direction. The infinitesimal generalization is correspondent to heat conduction in nature. Hence, the gradient method~(\ref{eqn:gradient}) is viewed as the implement in computer or optimization simulating the phenomena in the real nature. The ODE~(\ref{eqn:2nd_system}) is the governing equation for the heavy ball motion with friction.  The infinitesimal generalization is correspondent to chord vibration in nature.  Hence, the momentum method~(\ref{eqn:momentum}) and the Nesterov's accelerated gradient method~(\ref{eqn:nesterov}) are viewed as the update version implement in computer or optimization by use of setting the friction force parameter $\gamma_{t}$.
     
Furthermore, we can view the three methods above as the thought for dissipating energy implemented in the computer. The unknown objective function in black box model can be viewed as the potential energy.  Hence, the initial energy is from the potential function $f(x_{0})$ at any position $x_{0}$ to the minimization value $f(x^{\star})$ at the position $x^{\star}$. The total energy is combined with the kinetic energy and the potential energy. The key observation in this paper is that we find the kinetic energy, or the velocity, is observable and controllable variable in the optimization process. In other words, we can compare the velocities in every step to look for local minimum in the computational process or re-set them to zero to arrive to artificially dissipate energy. 

Let us introduce firstly the governing motion equation in a conservation force field, that we use in this paper, for comparison as below, 
\begin{equation}
\label{eqn:conservation}
\left\{\begin{aligned}
& \ddot{x} = - \nabla f(x) \\
&  x(0)= x_{0}, \;\dot{x}(0) =0. 
\end{aligned}\right.
\end{equation}
The concept of phase space, developed in the late 19th century, usually consists of all possible values of position and momentum variables. The governing motion equation in a conservation force field~(\ref{eqn:conservation}) can be rewritten as
\begin{equation}
\label{eqn:conphase}
\left\{\begin{aligned}
& \dot{x} = v \\
& \dot{v} = -\nabla f(x) \\
&  x(0)= x_{0}, \;v(0) =0 .
\end{aligned}\right.
\end{equation}

In this paper, we implement our discrete strategy with the utility of the observability and controllability of the velocity, or the kinetic energy, as well as artificially dissipating energy for two directions as below,   
\begin{itemize}
\item To look for local minima in non-convex function or global minima in convex function, the kinetic energy, or the norm of the velocity, is compared with that in the previous step,  it will be re-set to zero until it becomes larger no longer. 

\item To look for global minima in non-convex function, an initial larger velocity $v(0) = v_{0}$ is implemented at the any initial position $x(0) = x_{0}$. A ball is implemented with~(\ref{eqn:conphase}), the local maximum of the kinetic energy  is recorded to discern how many local minima exists along the trajectory. Then implementing the strategy above to find the minimum of all the local minima.
\end{itemize}
For implementing our thought in practice, we utilize the scheme in the numerical method for Hamiltonian system, the symplectic Euler method. We remark that a more accuracy version is the St\"{o}rmer-Verlet method for practice. 

\subsection{An Analytical Demonstration For Intuition}
For a simple 1-D function with ill-conditioned Hessian, $f(x) = \frac{1}{200} x^{2}$ with the initial position at $x_{0} = 1000$.  The solution and the function value along the solution for~(\ref{eqn:gradient_system}) are given by
\begin{numcases}{}
\label{eqn:grad1} x(t) = x_{0}e^{-\frac{1}{100}t}      \\
\label{eqn:grad2} f(x(t)) = \frac{1}{200}x_{0}^{2} e^{-\frac{1}{50}t}. 
\end{numcases}
The solution and the function value along the solution for~(\ref{eqn:2nd_system}) with the optimal friction parameter $\gamma_{t} = \frac{1}{5}$ are
\begin{numcases}{}
\label{eqn:mom1} x(t) = x_{0}\left(1 + \frac{1}{10}t \right)e^{-\frac{1}{10}t}      \\
\label{eqn:mom2} f(x(t)) = \frac{1}{200}x_{0}^{2} \left(1 + \frac{1}{10}t \right)^{2} e^{-\frac{1}{5}t}.
\end{numcases}
The solution and the function value along the solution for~(\ref{eqn:conphase}) are
\begin{numcases}{}
\label{eqn:con1} x(t) = x_{0}\cos\left(\frac{1}{10}t\right)\quad \text{and}\quad  v(t) = x_{0}\sin\left(\frac{1}{10}t\right)     \\
\label{eqn:con2} f(x(t)) = \frac{1}{200}x_{0}^{2} \cos^{2}\left(\frac{1}{10}t\right)
\end{numcases}
stop at the point that $|v|$ arrive maximum. Combined with~(\ref{eqn:grad2}),~(\ref{eqn:mom2}) and~(\ref{eqn:con2}) with stop at the point that $|v|$ arrive maximum, the function value approximating $f(x^{\star})$ are shown as below,
\begin{figure}[H]
\begin{minipage}[t]{0.5\linewidth}
\centering
\includegraphics[width=3.2in]{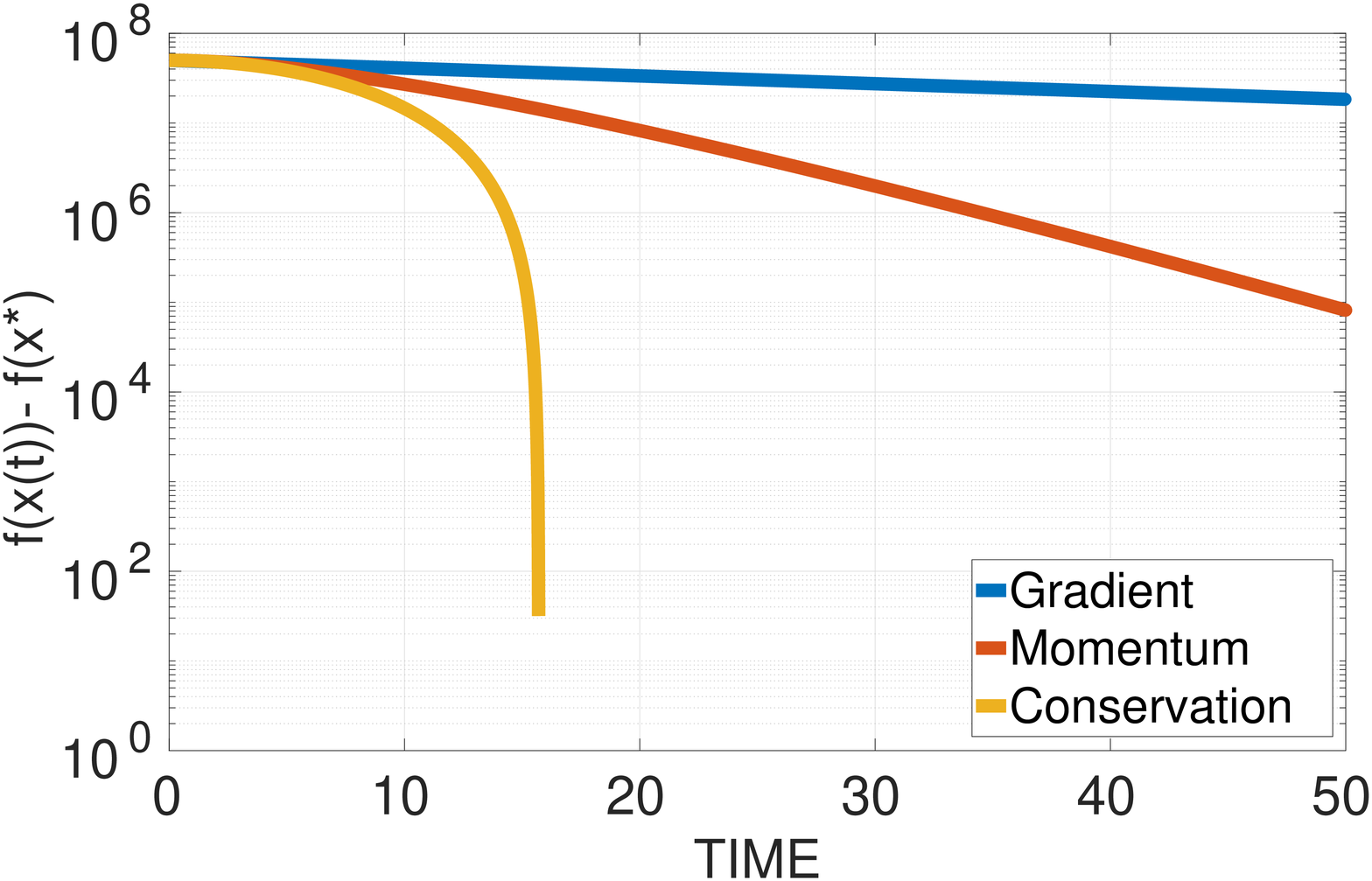}
\label{fig:side:a}
\end{minipage}%
\begin{minipage}[t]{0.5\linewidth}
\centering
\includegraphics[width=3.2in]{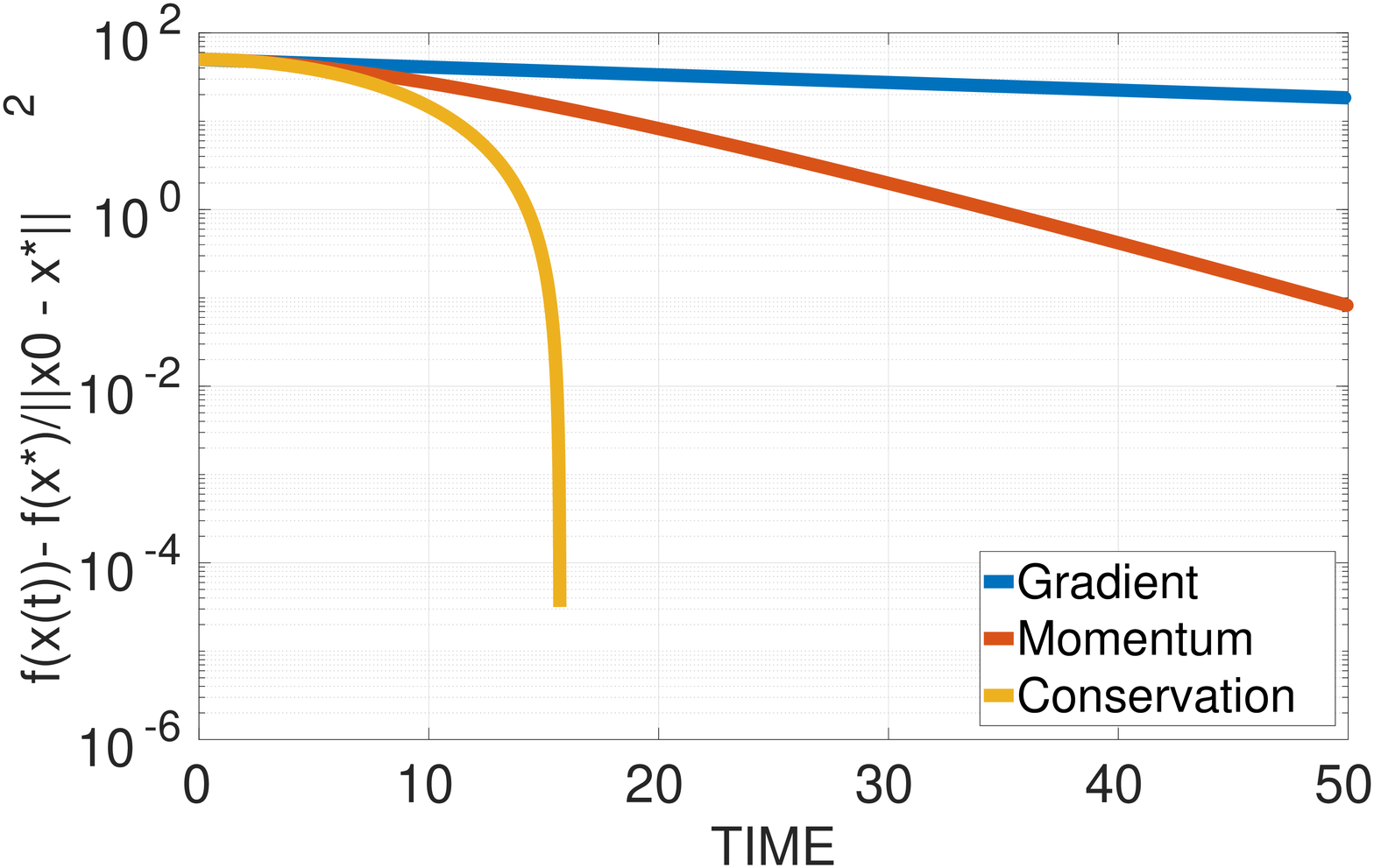}
\label{fig:side:b}
\end{minipage}
\caption{Minimizing $f(x) = \frac{1}{200}x^{2}$ by the analytical solution for~(\ref{eqn:grad2}),~(\ref{eqn:mom2}) and~(\ref{eqn:con2}) with stop at the point that $|v|$ arrive maximum, starting from $x_{0} = 1000$ and the numerical step size $\Delta t =0.01$.} 
\label{fig:analyse}
\end{figure}
From the analytical solution for local convex quadratic function with maximum eigenvalue $L$ and minimum eigenvalue $\mu$, in general, the step size by $\frac{1}{\sqrt{L}}$ for momentum method and Nesterov accelerated gradient method, hence the simple estimate for iterative times is approximately 
$$
n \sim \frac{\pi}{2}\sqrt{\frac{L}{\mu}}.
$$ 
hence, the iterative times $n$ is proportional to the reciprocal of the square root of minimal eigenvalue $\sqrt{\mu}$, which is essentially different from the convergence rate of the gradient method and momentum method.

The rest of the paper is organized as follows. Section~\ref{sec:related_work} summarize relevant existing works. In Section~\ref{sec:symplectic}, we propose the artificially dissipating energy algorithm, energy conservation algorithm and the combined algorithm based on the symplectic Euler scheme, and remark a second-order scheme --- the St\"{o}rmer-Verlet scheme . In Section~\ref{sec:local}, we propose the locally theoretical analysis for High-Speed converegnce. In section~\ref{sec:experimental}, we propose the experimental result for the proposed algorithms on strongly convex function, non-strongly convex function and nonconvex function in high-dimension. Section~\ref{sec:conclusion} proposes some perspective view for the proposed algorithms and two adventurous ideas based on the evolution of Newton Second Law --- fluid and quantum.

%% file: 02_related.tex
\section{Related Work}
\label{sec:related_work}

The history of gradient method for convex optimization can be back to the time of Euler and Lagrange. However, since it is relatively cheaper to only calculate for first-order information, this simplest and earliest method is still active in machine learning and nonconvex optimization, such as the recent work~\citep{ge2015escaping, anandkumar2016efficient, lee2016gradient, hardt2016gradient}. The natural speedup algorithms are the momentum method first proposed in~\citep{polyak1964some} and Nesterov accelerated gradient method first proposed in~\citep{nesterov1983method} and an improved version~\citep{nesterov1988general}. A acceleration algorithm similar as Nesterov accelerated gradient method, named as FISTA, is designed to solve composition problems~\citep{beck2009fast}.  A related comprehensive work is proposed in~\citep{bubeck2015convex}.  

The original momentum method, named as Polyak heavy ball method, is from the view of ODE in~\citep{polyak1964some}, which contains extremely rich physical intuitive ideas and mathematical theory. An extremely important work in application on machine learning is the backpropagation learning with momentum~\citep{rumelhart1988learning}. Based on the thought of ODE, a lot of understanding and application on the momentum method and Nesterov accelerated gradient methods have been proposed. In~\citep{sutskever2013importance}, a well-designed random initialization with momentum parameter algorithm is proposed to train both DNNs and RNNs. A seminal deep insight from ODE to understand the intuition behind Nesterov scheme is proposed in~\citep{su2014differential}. The understanding for momentum method based on the variation perspective is proposed on~\citep{wibisono2016variational}, and the understanding from Lyaponuv analysis is proposed in~\citep{wilson2016lyapunov}.  From the stability theorem of ODE, the gradient method always converges to local minima in the sense of almost everywhere is proposed in~\citep{lee2016gradient}.  Analyzing and designing iterative optimization algorithms built on integral quadratic constraints from robust control theory is proposed in~\citep{lessard2016analysis}. 

Actually the ``high momentum'' phenomenon has been firstly observed in~\citep{o2015adaptive} for a restarting adaptive accelerating algorithm, and also the restarting scheme is proposed by~\citep{su2014differential}. However, both works above utilize restarting scheme for an auxiliary tool to accelerate the algorithm based on friction. With the concept of phase space in mechanics, we observe that the kinetic energy, or velocity, is controllable and utilizable parameter to find the local minima. Without friction term, we can still find the local minima only by the velocity parameter.  Based on this view, the algorithm is proposed very easy to practice and propose the theoretical analysis. Meanwhile, the thought can be generalized to nonconvex optimization to detect local minima along the trajectory of the particle.

%% file: 03_symplectic.tex
\section{Symplectic Scheme and Algorithms}
\label{sec:symplectic}
In this chapter, we utilize the first-order symplectic Euler scheme from numerically solving Hamiltonian system as below
\begin{equation}
\label{eqn:symplectic_euler}
\left\{\begin{aligned}
         & x_{k+1} = x_{k} + h v_{k+1} \\
         & v_{k+1} = v_{k} - h \nabla f(x_{k})
        \end{aligned}\right.
\end{equation}
to propose the corresponding artifically dissipating energy algorithm to find the global minima for convex function, or local minima in non-convex function. Then by the observability of the velocity, we propose the energy conservation algorithm for detecting local minima along the trajectory. Finally, we propose a combined algorithm to find better local minima between some local minima.
\begin{rem}
In all the algorithms below, the symplectic Euler scheme can be taken place by the St\"{o}rmer-Verlet scheme, i.e.  
\begin{equation}
\label{eqn:symplectic_vertex}
\left\{\begin{aligned}
        & v_{k+1/2} = v_{k} - \frac{h}{2}\nabla f(x_{k}) \\
        & x_{k+1} = x_{k} + h v_{k + 1/2} \\
        & v_{k+1} = v_{k + 1/2} - \frac{h}{2} \nabla f(x_{k+1})
        \end{aligned}\right. 
\end{equation}
which works perfectly better than the symplectic scheme even if doubling step size and keep the left-right symmetry of the Hamiltonian system. The St\"{o}rmer-Verlet scheme is the natural discretization for 2nd-order ODE
\begin{equation}
\label{eqn:discrete}
x_{k+1} - 2x_{k} + x_{k-1} = -h^{2}\nabla f(x_{k})
\end{equation}
which is named as leap-frog scheme in PDEs. We remark that the discrete scheme~(\ref{eqn:discrete}) is different from the finite difference approximation by the forward Euler method to analyze the stability of 2nd ODE in~\citep{su2014differential}, since the momentum term is biased.  
\end{rem}
\subsection{The Artifically Dissipating Energy Algorithm}
Firstly, the artificially dissipating energy algorithm based on~(\ref{eqn:symplectic_euler}) is proposed as below.
\begin{algorithm}[H]
\begin{algorithmic}[1]
\State{Given a starting point $x_{0} \in \mathbf{dom}(f)$}
\State{Initialize the step length $h$, $\text{maxiter}$, and the velocity variable $v_{0} = 0$}
\State{Initialize the iterative variable $v_{iter} = v_{0}$}
\While{$\|\nabla f(x)\| > \epsilon$\;and\; $k < \text{maxiter}$}
\State{Compute $v_{iter}$ from the below equation in~(\ref{eqn:symplectic_euler})}
\If{$\|v_{iter}\| \leq \|v\|$}
         \State{$v = 0$}
      \Else
         \State{$v = v_{iter}$}
\EndIf
  
\State{Compute $x$ from the above equation in~(\ref{eqn:symplectic_euler})}
\State{$x_k = x$;}
\State{$f(x_{k}) = f(x)$;}
\State{$k = k + 1$;}
\EndWhile
\end{algorithmic}
\caption{Artifically Dissipating Energy Algorithm}
\label{alg:local}
\end{algorithm}

\begin{rem}
In the actual algorithm~\ref{alg:local}, the codes in line $15$ and $16$ are not need in the while loop in order to  speed up the computation.
\end{rem}

\subsubsection{A Simple Example For Illustration}
Here, we use a simple convex quadratic function with ill-conditioned eigenvalue for illustration as below,
\begin{equation}
\label{eqn:artifical1}
f(x_1, x_{2}) = \frac{1}{2}\left(x_{1}^{2} + \alpha x_{2}^{2}\right),
\end{equation}
of which the maximum eigenvalue is $L=1$ and the minimum eigenvalue is $\mu = \alpha$.  Hence the scale of the step size for~(\ref{eqn:artifical1}) is 
$$
\frac{1}{L} = \sqrt{\frac{1}{L}} = 1.
$$
In figure~\ref{fig:common}, we demonstrate the convergence rate of gradient method, momentum method, Nesterov accelerated gradient method and artifically dissipating energy  method with the common step size $h = 0.1$ and $h = 0.5$, where the optimal friction parameter for momentum method and Nesterov accelerated gradient method $\gamma = \frac{1-\sqrt{\alpha}}{1+\sqrt{\alpha}}$ with $\alpha = 10^{-5}$.  A further result for comparison with the optimal step size in gradient method $h = \frac{2}{1+\alpha}$, the momentum method $h = \frac{4}{(1+\sqrt{\alpha})^{2}}$, and Nesterov accelerated gradient method with $h=1$ and the artifically disspating energy  method with $h=0.5$ shown in figure~\ref{fig:diff}.
\begin{figure}[H]
\begin{minipage}[t]{0.5\linewidth}
\centering
\includegraphics[width=3.2in]{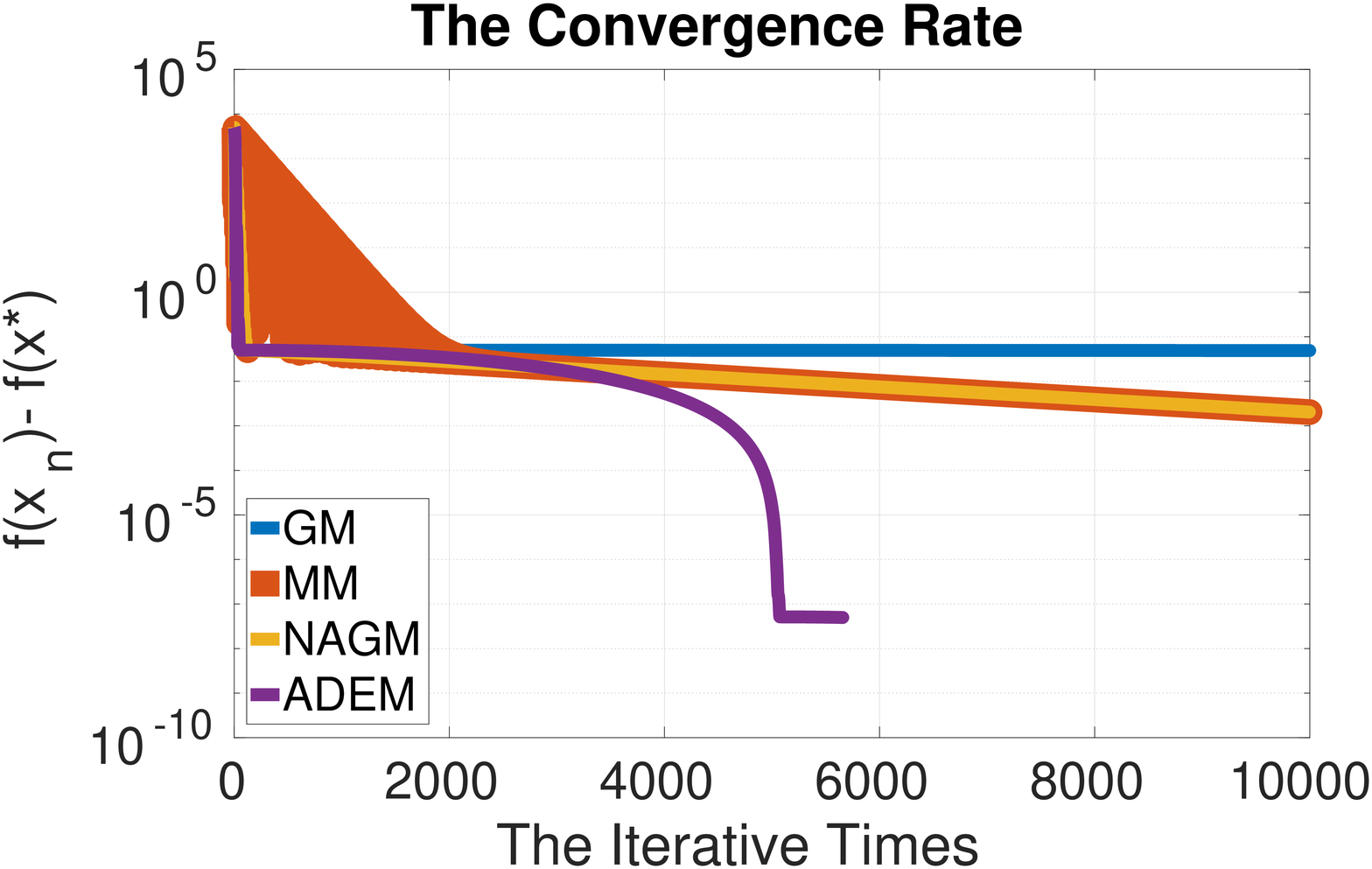}
\end{minipage}%
\begin{minipage}[t]{0.5\linewidth}
\centering
\includegraphics[width=3.2in]{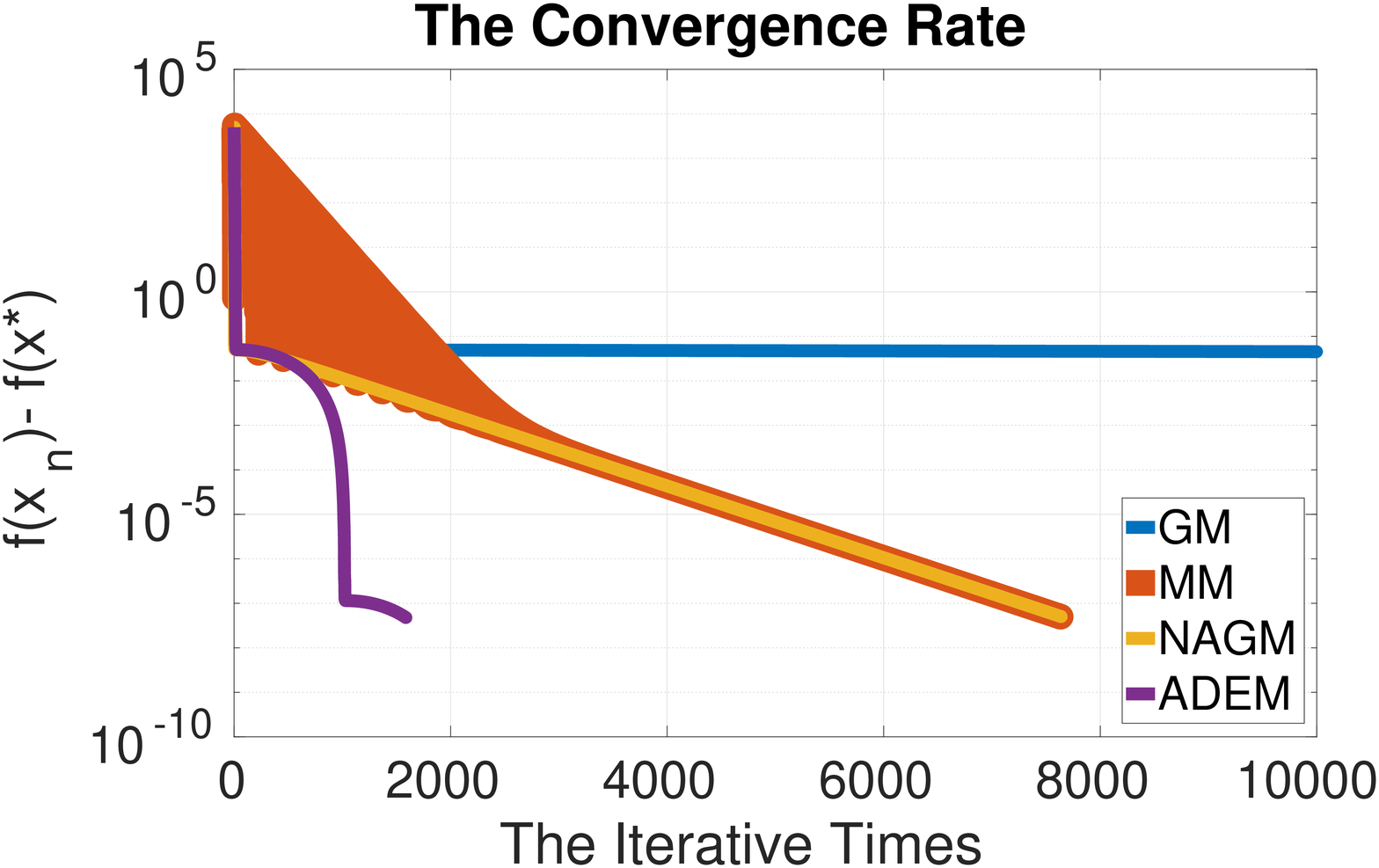}
\end{minipage}
\caption{Mimimize the function in~(\ref{eqn:artifical1}) for artificially dissipating energy algorithm comparing with gradient method, momentum method and Nesterov accelerated gradient method with stop criteria $\epsilon = 1e-6$. The Step size: Left: $h=0.1$; Right: $h=0.5$.} 
\label{fig:common}
\end{figure}

\begin{figure}[H]
\begin{minipage}[t]{0.5\linewidth}
\centering
\includegraphics[width=3.2in]{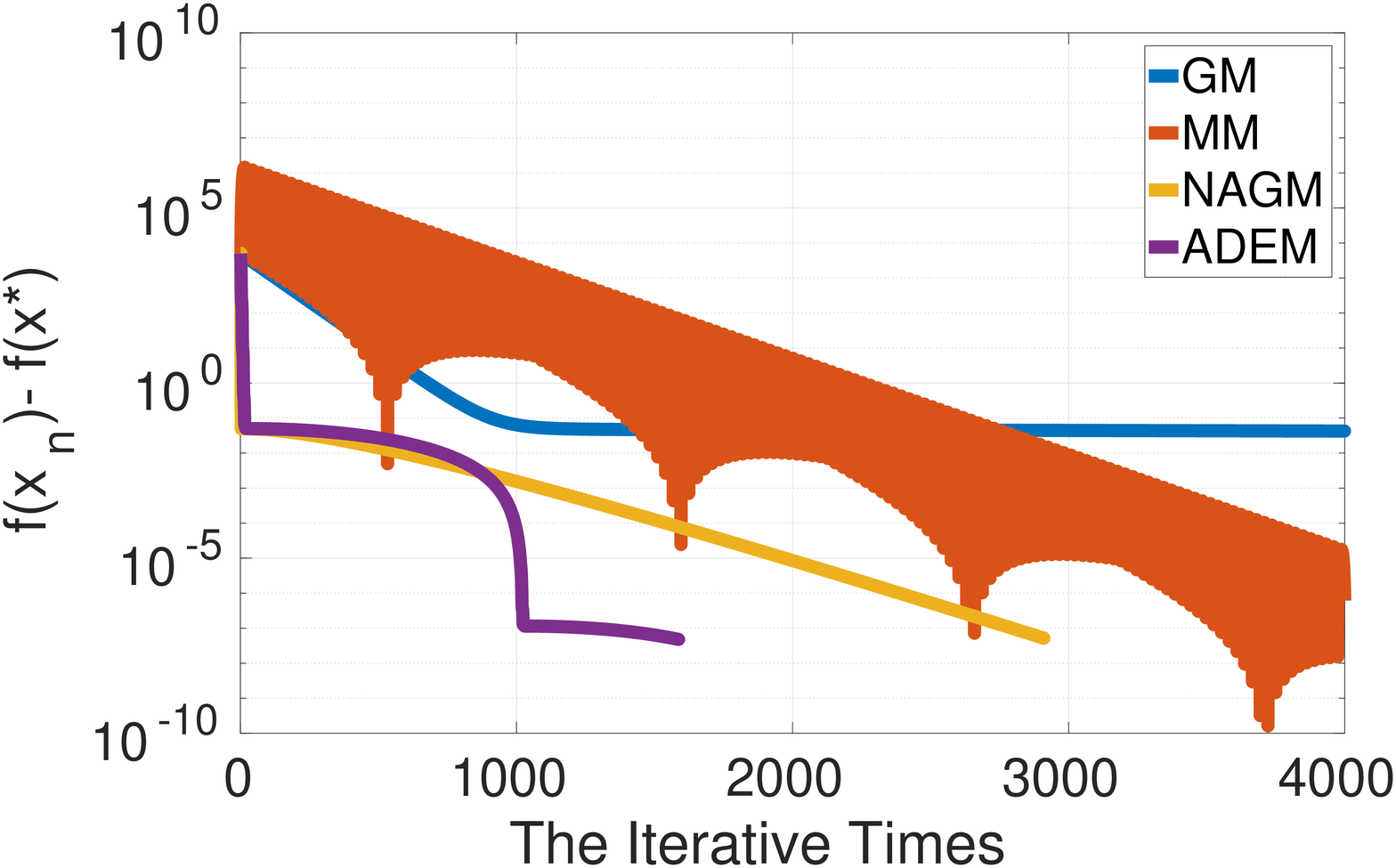}
\end{minipage}%
\begin{minipage}[t]{0.5\linewidth}
\centering
\includegraphics[width=3.2in]{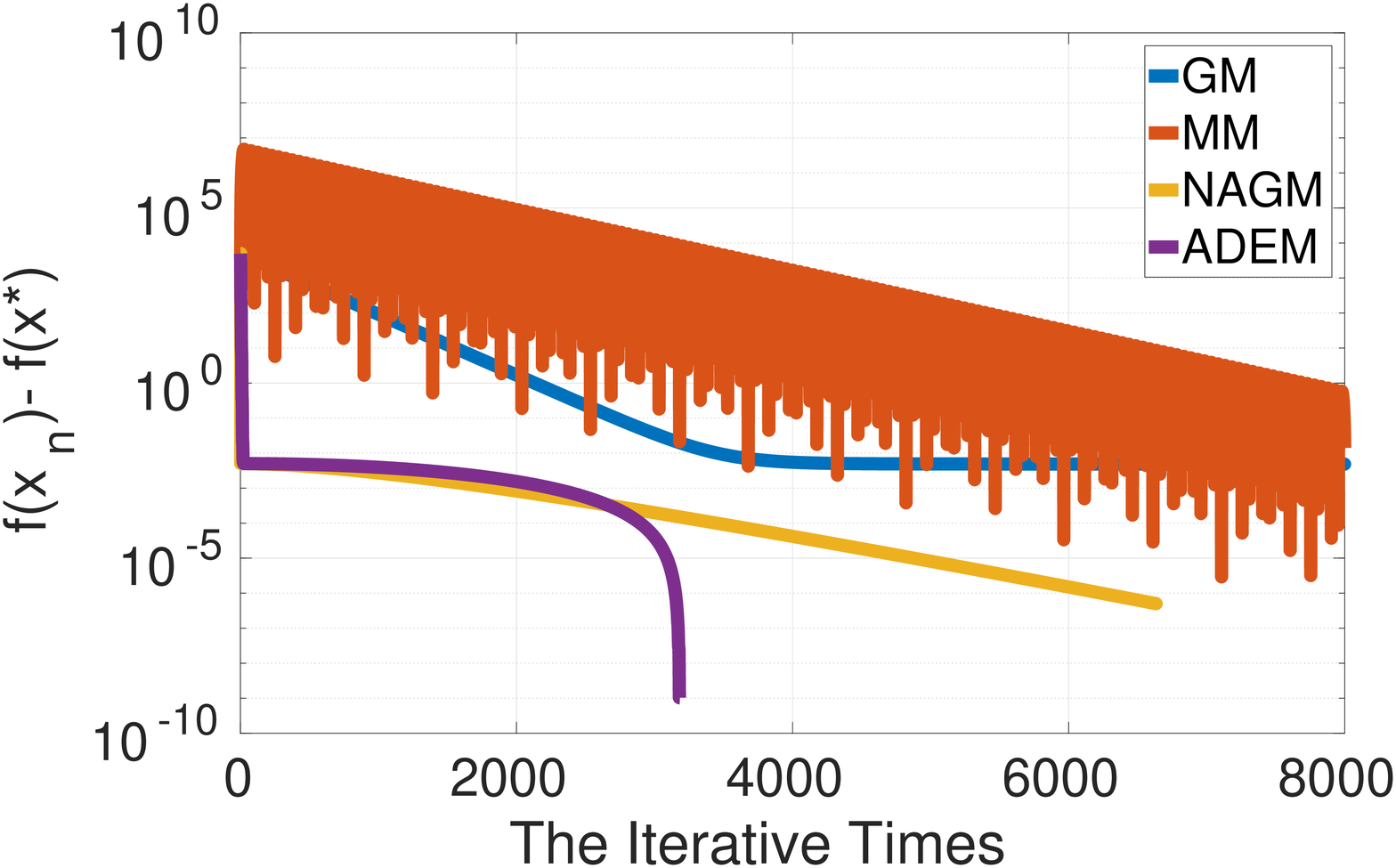}
\end{minipage}
\caption{Mimimize the function in~(\ref{eqn:artifical1}) for artificially dissipating energy algorithm comparing with gradient method, momentum method and Nesterov accelerated gradient method with stop criteria $\epsilon = 1e-6$. The Coefficient $\alpha$: Left: $\alpha=10^{-5}$; Right: $\alpha=10^{-6}$.} 
\label{fig:diff}
\end{figure}

With the illustrative convergence rate, we need to learn the trajectory. Since the trajectories of all the four methods are so narrow in ill-condition function in~(\ref{eqn:artifical1}), we use a relatively good-conditioned function to show it as $\alpha = \frac{1}{10}$ in figure~\ref{fig:trajectory}.  
\begin{figure}[H]
\centering
\includegraphics[width=3.2in]{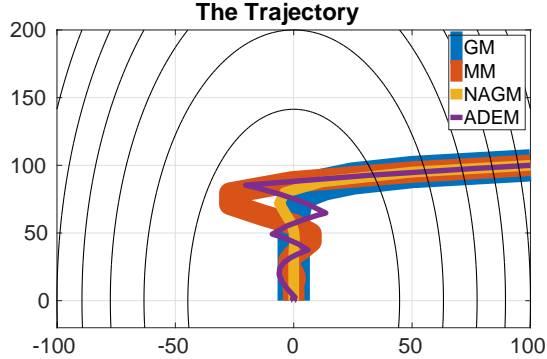}
\caption{The trajectory for gradient method, momentum method, Nesterov accelerated method and artifically dissipating energy method for the function~(\ref{eqn:artifical1}) with $\alpha = 0.1$.} 
\label{fig:trajectory}
\end{figure}
A clear fact in figure~\ref{fig:trajectory} shows that the gradient correction decrease the oscillation to comparing with momentum method. A more clear observation is that artificially dissipating method owns the same property with the other three method by the law of nature, that is, if the trajectory come into the local minima in one dimension will not leave it very far. However, from figure~\ref{fig:common} and figure~\ref{fig:diff}, the more rapid convergence rate from artificially dissipating energy method has been shown. 
\subsection{Energy Conservation Algorithm For Detecting Local Minima}
Here,  the energy conservation algorithm based on~(\ref{eqn:symplectic_euler}) is proposed as below.
\begin{algorithm}[H]
\begin{algorithmic}[1]
\State{Given a starting point $x_{0} \in \mathbf{dom}(f)$}
\State{Initialize the step size $h$ and the maxiter}
\State{Initialize the velocity $v_{0}>0$ and compute $f(x_{0})$}
\State{Compute the velocity $x_1$ and $v_{1}$ from the equation~(\ref{eqn:symplectic_euler}), and compute $f(x_1)$ }
\For{$k= 1: n$}
      \State{Compute $x_{k+1}$ and $v_{k+1}$ from~(\ref{eqn:symplectic_euler})}
      \State{Compute $f(x_{k+1})$}
      \If{$\|v_{k}\| \geq \|v_{k+1}\|$ and $\|v_{k}\| \geq \|v_{k-1}\|$ }
      \State{Record the position $x_{k}$}
      \EndIf
\EndFor 
\end{algorithmic}
\caption{Energy Conservation Algorithm}
\label{alg:global}
\end{algorithm}
\begin{rem}
In the algorithm~\ref{alg:global}, we can set $v_{0} > 0$ such that the total energy large enough to climb up some high peak. Same as the algorithm~\ref{alg:local}, the function value $f(x)$ is not need in the while loop in order to  speed up the computation.
\end{rem}

\subsubsection{The Simple Example For Illustration}
Here, we use the non-convex function for illustration as below,
\begin{equation}
\label{eqn:noncon1}
f(x) = \left\{\begin{aligned}
                 & 2\cos(x),     && x \in [0,2\pi]      \\
                 & \cos(x) + 1, && x \in [2\pi, 4\pi]  \\
                 & 3\cos(x) -1, && x \in [4\pi, 6\pi]  
                 \end{aligned}\right.
\end{equation}
which is the 2nd-order smooth function but not 3rd-order smooth. The maximum eigenvalue can be calculated as below
$$
\max\limits_{x\in [0,6\pi]} |f''(x)| = 3.
$$
then, the step length is set $h \sim \sqrt{\frac{1}{L}}$. We illustrate that the algorithm~\ref{alg:global} simulate the trajectory and find the local minima in figure~\ref{fig:noncon1}. 
\begin{figure}[H]
\begin{minipage}[t]{0.5\linewidth}
\centering
\includegraphics[width=3.2in]{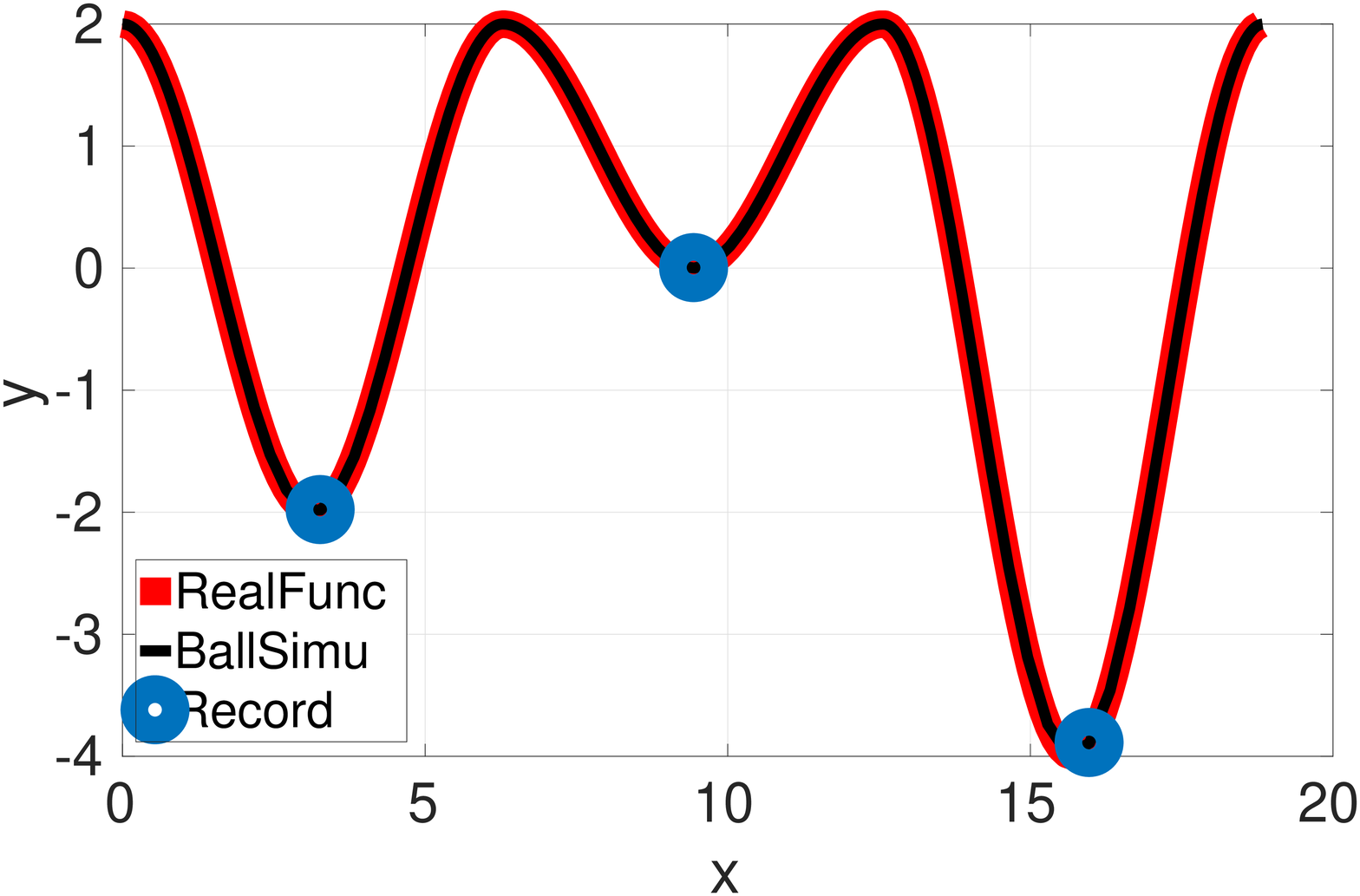}
\end{minipage}%
\begin{minipage}[t]{0.5\linewidth}
\centering
\includegraphics[width=3.2in]{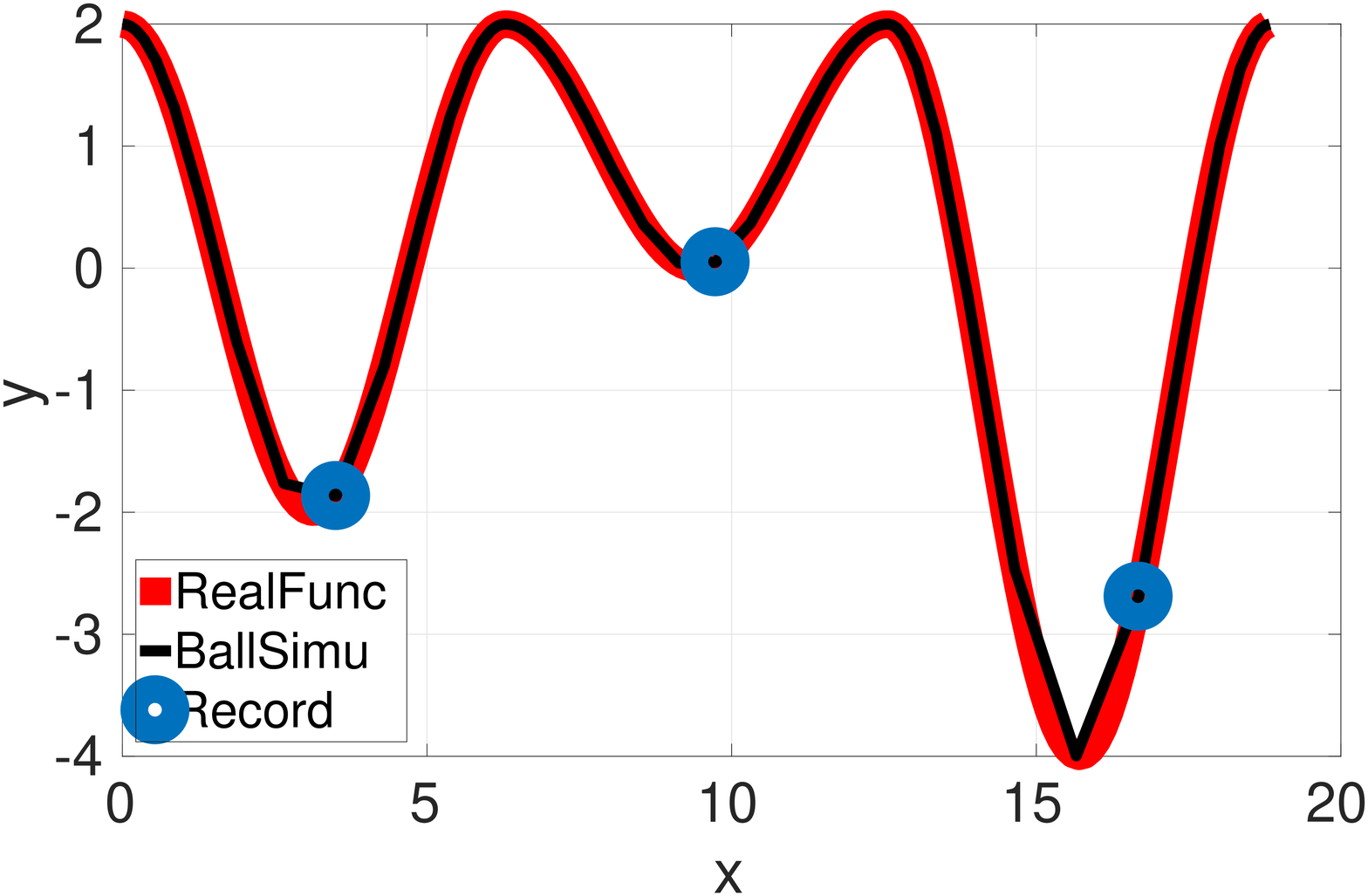}
\end{minipage}
\caption{The Left: the step size $h=0.1$ with $180$ iterative times. The Right: the step size $h = 0.3$ with $61$ iterative times.} 
\label{fig:noncon1}
\end{figure}

Another 2D potential function is shown as below,
\begin{equation}
\label{eqn:noncon2}
f(x_1,x_2) = \frac{1}{2} \left[\left(x_1 - 4\right)^{2} + \left(x_2 - 4\right)^{2} + 8\sin(x_1 + 2x_2)\right].
\end{equation}
which is the  smooth function with domain in $(x_{1},x_{2}) \in [0,8]\times [0,8]$. The maximum eigenvalue can be calculated as below
$$
\max\limits_{x\in [0,6\pi]} |\lambda(f''(x))| \geq 16.
$$
then, the step length is set $h \sim \sqrt{\frac{1}{L}}$. We illustrate that the algorithm~\ref{alg:global} simulate the trajectory and find the local minima in figure~\ref{fig:noncon2}. 

\begin{figure}[H]
\begin{minipage}[t]{0.5\linewidth}
\centering
\includegraphics[width=3.2in]{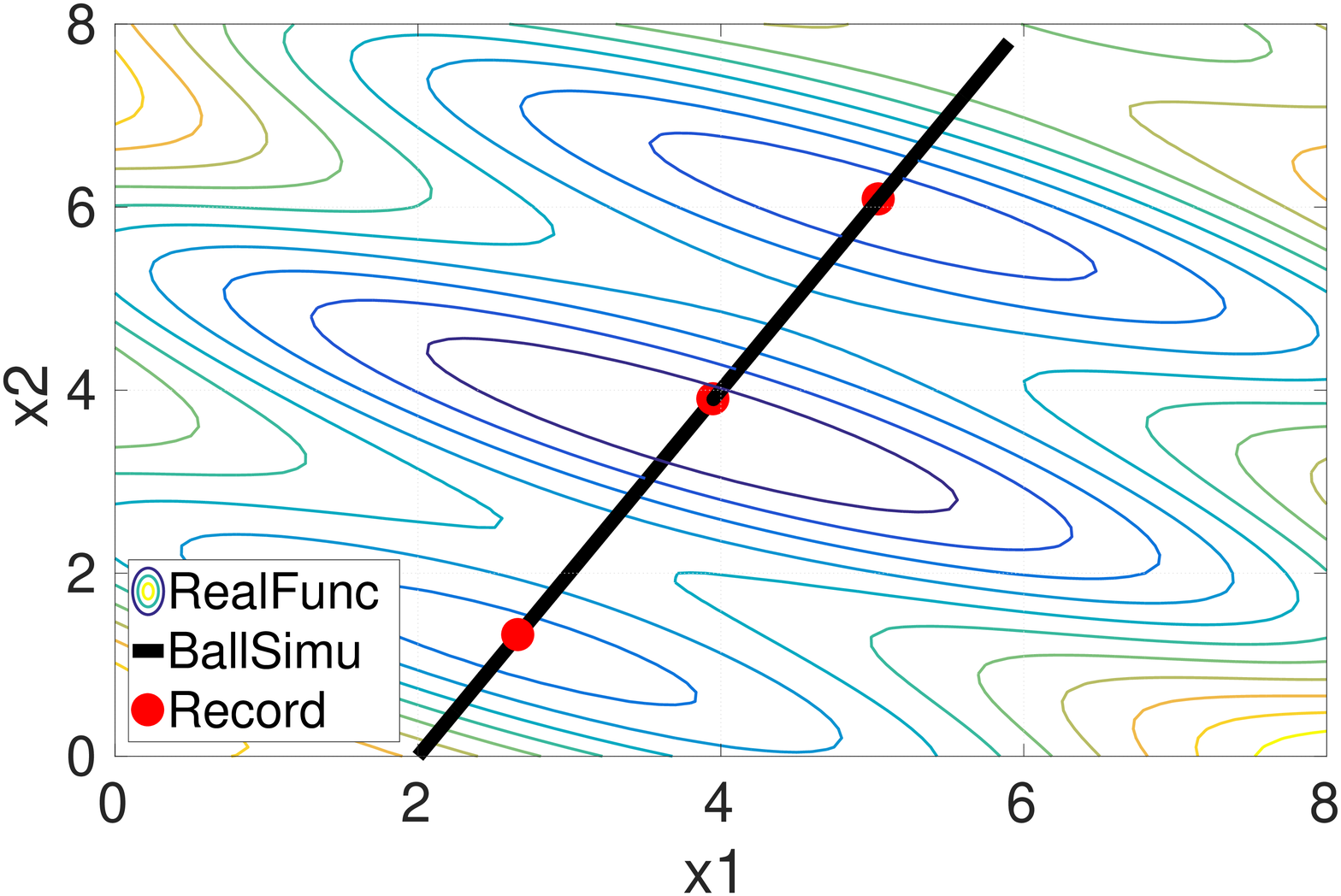}
\end{minipage}%
\begin{minipage}[t]{0.5\linewidth}
\centering
\includegraphics[width=3.2in]{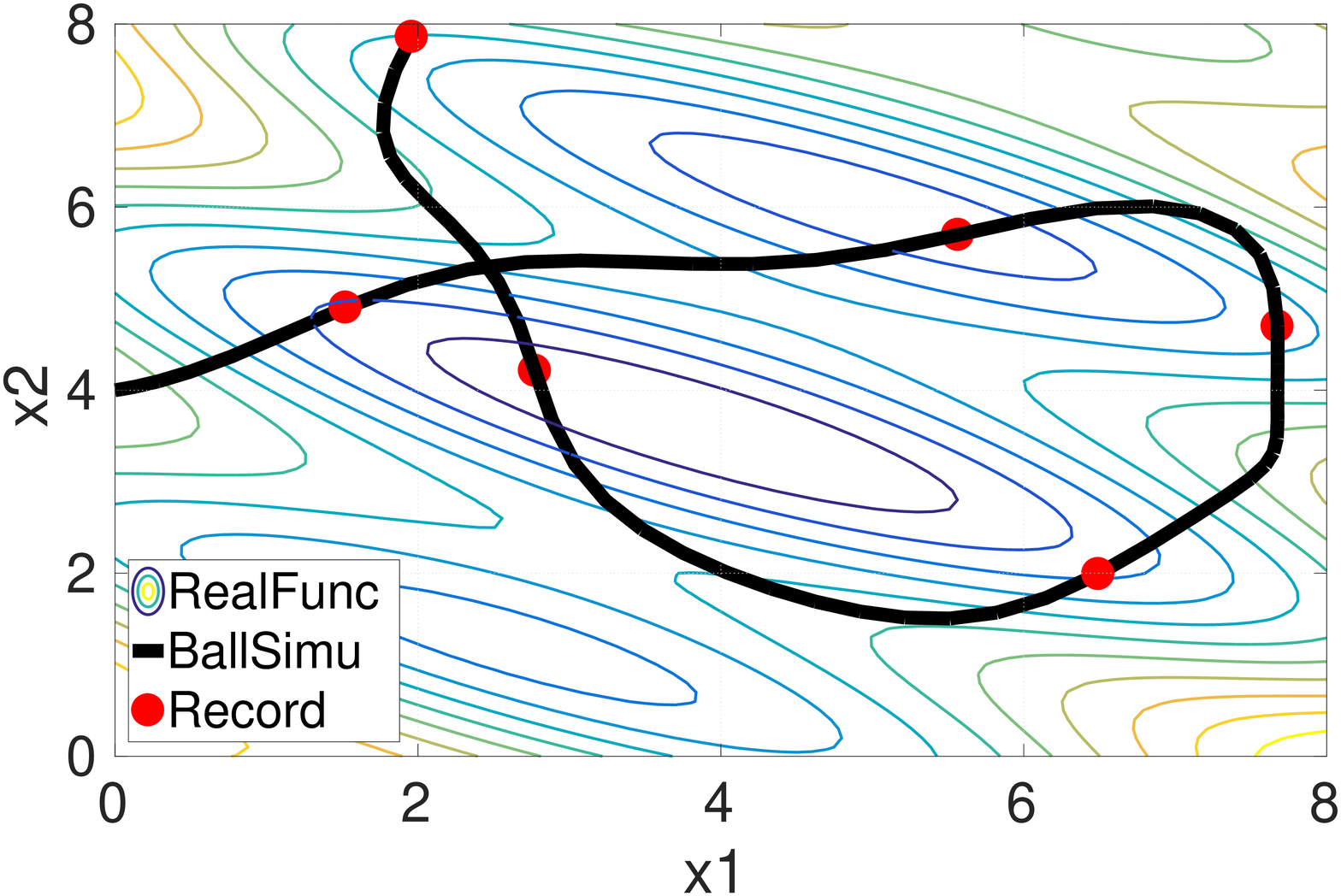}
\end{minipage}
\caption{The common step size is set $h = 0.1$. The Left: the position at $(2,0)$ with $23$ iterative times. The Right: the position at $(0,4)$ with $62$ iterative times.} 
\label{fig:noncon2}
\end{figure}
\begin{rem}
We point out that for the energy conservation algorithm for detecting local minima along the trajectory cannot detect saddle point in the sense of almost everywhere, since the saddle point in original function $f(x)$ is also a saddle point for the energy function $H(x,v) = \frac{1}{2}\|v\|^{2} + f(x)$. The proof process is fully the same in~\citep{lee2016gradient}. 
\end{rem}
\subsection{Combined Algorithm}
Finally, we propose the comprehensive algorithm combining the artificially dissipating energy algorithm (algorithm~\ref{alg:local}) and the energy conservation algorithm (algorithm~\ref{alg:global}) to find global minima.
\begin{algorithm}[H]
\begin{algorithmic}[1]
\State{Given some starting points $x_{0,i} \in \mathbf{dom}(f)$ with $i = 1, \ldots,n$}
\State{Implement algorithm~\ref{alg:global} detecting the position there exists local minima, noted as $x_{j}$ with $j = 1, \ldots,m$}
\State{Implement algorithm~\ref{alg:local} from the result on line $2$ finding the local minima, noted as $x_{k}$ with $k=1,\ldots, l$}
\State{Comparison of $f(x_{k})$ with $k=1,\ldots,l$ to find global minima.}
\end{algorithmic}
\caption{Combined Algorithm}
\label{alg:combined}
\end{algorithm}
\begin{rem}
We remark that the combined algorithm (algorithm~\ref{alg:combined}) cannot guarantee to find global minima if the initial position is not ergodic. The tracking local minima is dependent on the trajectory. However, the time of computation and precision based on the proposed algorithm is far better than the large sampled gradient method. Our proposed algorithm first makes the global minima found become possible. 
\end{rem}

%% file: 04_local.tex
\section{An Asymptotic Analysis for The Phenomena of Local High-Speed Convergence}
\label{sec:local}
In this section, we analyze the phenomena of high-speed convergence shown in figure~\ref{fig:analyse}, figure~\ref{fig:common} and figure~\ref{fig:diff}. Without loss of generality, we use the translate transformation $y_{k} = x_{k} -x^{\star}$ ($x^{\star}$ is the point of local minima) and $v_{k}=v_{k}$ into~(\ref{eqn:symplectic_euler}), shown as below,
\begin{equation}
\label{eqn:equiv_sym}
\left\{
\begin{aligned}
&y_{k+1} = y_{k} + hv_{k+1} \\
&v_{k+1} = v_{k} - h\nabla f(x^{\star} + y_{k}),
\end{aligned}
\right.
\end{equation}
the locally linearized scheme of which is given as below,
\begin{equation}
\label{eqn:equiv_linear}
\left\{
\begin{aligned}
&y_{k+1} = y_{k} + hv_{k+1} \\
&v_{k+1} = v_{k} - h\nabla^{2} f(x^{\star})y_{k}.
\end{aligned}
\right.
\end{equation}
\begin{rem}
The local linearized analysis is based on the stability theorem in finite dimension, the invariant stable manifold theorem and Hartman-Grobman linearized map theorem~\citep{hartman1982ordinary}. The thought is firstly used in~\citep{polyak1964some} to estimate the local convergence of momentum method. And in the paper~\citep{lee2016gradient}, the thought is used to exclude the possiblity to converegnce to saddle point. However, the two theorems above belong to the qualitative theorem of ODE. Hence, the linearized scheme~(\ref{eqn:equiv_linear}) is only an approximate estimate for the original scheme~(\ref{eqn:equiv_sym}) locally.
\end{rem}
\subsection{Some Lemmas For The Linearized Scheme}
Let $A$ be the positive-semidefinite and symmetric matrix to represent $\nabla^{2}f(x^{\star})$ in~(\ref{eqn:equiv_linear}).
 
\begin{lem}
\label{lem:equiv} The numerical scheme, shown as below
\begin{equation}
\label{eqn:equiv}
\begin{pmatrix}
       y_{k+1} \\
       v_{k+1}  
\end{pmatrix} = \begin{pmatrix}
       I -h^{2}A & hI \\
       -hA       & I 
\end{pmatrix}\begin{pmatrix}
       y_{k} \\
       v_{k}  
\end{pmatrix}
\end{equation}
is equivalent to the linearized symplectic-Euler scheme~(\ref{eqn:equiv_linear}), where we note that the linear transformation is 
\begin{equation}
\label{eqn:equivmatrix}
M
=\begin{pmatrix}
       I -h^{2}A & hI \\
       -hA       & I 
\end{pmatrix}.
\end{equation}
\end{lem}
\begin{proof}
$$
\begin{aligned}
&\begin{pmatrix}
I & -hI \\
0 &   I
\end{pmatrix}\begin{pmatrix}
       y_{k+1} \\
       v_{k+1}  
\end{pmatrix} = \begin{pmatrix}
I   &   0 \\
-hA &   I
\end{pmatrix}\begin{pmatrix}
       y_{k} \\
       v_{k}  
\end{pmatrix}
\Leftrightarrow \begin{pmatrix}
       y_{k+1} \\
       v_{k+1}  
\end{pmatrix} = \begin{pmatrix}
I-h^{2}A   &   hI \\
-hA &   I
\end{pmatrix}\begin{pmatrix}
       y_{k} \\
       v_{k}  
\end{pmatrix}
 \end{aligned}
$$
\end{proof}

\begin{lem}
\label{lem:symmetric}
For every $2n\times 2n$ matrix $M$ in~(\ref{eqn:equivmatrix}), there exists the orthogonal transformation $U_{2n\times 2n}$ such that the matrix $M$ is similar as below
\begin{equation}
\label{eqn:equiveigen}
U^{T}MU = \begin{pmatrix} T_{1} & & & \\ &T_{2} & & \\ & &\ddots & \\ & & & T_{n} \end{pmatrix} 
\end{equation}
where $T_{i}$ $(i=1,\ldots, n)$ is $2\times 2$ matrix with the form 
\begin{equation}
\label{eqn:2by2}
T_{i} = \begin{pmatrix}
        1-\omega_{i}^{2}h^{2} & h \\
        -\omega_{i}^{2}h      & 1
        \end{pmatrix}
\end{equation}
where $\omega_{i}^{2}$ is the eigenvalue of the matrix $A$.
\end{lem}
\begin{proof}
Let $\Lambda$ be the diagonal matrix with the eigenvalues of the matrix $A$ as below
$$
\Lambda =\begin{pmatrix} \omega^{2}_{1} & & & \\ &\omega^{2}_{2} & & \\ & &\ddots & \\ & & & \omega^{2}_{n} \end{pmatrix}.
$$
Since $A$ is positive definite and symmetric, there exists orthogonal matrix $U_{1}$ such that 
$$
U_{1}^{T}AU_{1} =  \Lambda.
$$
Let $\Pi$ be the permuation matrix satisfying 
$$
\Pi_{i,j} = \left\{
      \begin{aligned}
       &1, && j\;\text{odd},\;  i=\frac{j+1}{2} \\
       &1, && j\;\text{even},\; i=n+\frac{j}{2} \\
       &0, && \text{otherwise}
      \end{aligned}
      \right.   
$$
where $i$ is the row index and $j$ is the column index. Then, let $U = \mathbf{diag}(U_{1},U_{1})\Pi$, we have by conjugation
$$
\begin{aligned}
U^{T}MU & =\Pi^{T} \begin{pmatrix} U_{1}^{T} & \\ & U_{1}^{T}\end{pmatrix} \begin{pmatrix}I-h^{2}A & hI \\ -hA & I \end{pmatrix} \begin{pmatrix} U_{1} & \\ & U_{1}\end{pmatrix}\Pi \\
        & = \Pi^{T} \begin{pmatrix} I-h^2\Lambda & hI \\ -h\Lambda & I \end{pmatrix} \Pi\\
        & = \begin{pmatrix} T_{1} & & & \\ &T_{2} & & \\ & &\ddots & \\ & & & T_{n} \end{pmatrix}  
\end{aligned}
$$
\end{proof}
From Lemma~\ref{lem:symmetric}, we know that the equation~(\ref{eqn:equiv}) can be written as the equivalent form
\begin{equation}
\label{eqn:equivtran}
\begin{pmatrix}
(U_{1}^{T}y)_{k+1,i} \\
(U_{1}^{T}v)_{k+1,i}
\end{pmatrix} =T_{i} \begin{pmatrix}
(U_{1}^{T}y)_{k,i} \\
(U_{1}^{T}v)_{k,i}
\end{pmatrix}= \begin{pmatrix}
                        1-\omega_{i}^{2}h^{2} & h \\
                          -\omega_{i}^{2}h        &1
                         \end{pmatrix}\begin{pmatrix}
(U_{1}^{T}y)_{k,i} \\
(U_{1}^{T}v)_{k,i}
\end{pmatrix}
\end{equation}
where $i = 1, \ldots, n$.
\begin{lem}
For any step size $h$ satisfying $0<h\omega_{i}<2$, the eigenvalues of the matrix $T_{i}$ are complex with absolute value $1$.
\end{lem}
\begin{proof} For $i = 1, \ldots, n$, we have
$$
\left|\lambda I - T_{i}\right| = 0\Leftrightarrow \lambda_{1,2} = 1-\frac{h^{2}\omega_{i}^{2}}{2} \pm h\omega_{i}\sqrt{1 - \frac{h^{2}\omega_{i}^{2}}{4}}.
$$
\end{proof}
Let $\theta_{i}$ and $\phi_{i}$ for $i = 1,\ldots,n$ for the new coordinate variables as below
\begin{equation}
\label{eqn:coortran}
\left\{\begin{aligned}
       & \cos\theta_{i} = 1-\frac{h^{2}\omega_{i}^{2}}{2} \\
       & \sin\theta_{i} = h\omega_{i}\sqrt{1 - \frac{h^{2}\omega_{i}^{2}}{4}}
       \end{aligned}\right.,\qquad\quad
\left\{\begin{aligned}
       & \cos\phi_{i} = \frac{h\omega_{i}}{2} \\
       & \sin\phi_{i} = \sqrt{1 - \frac{h^{2}\omega_{i}^{2}}{4}}
       \end{aligned}\right.
\end{equation}
In order to make $\theta_{i}$ and $\phi_{i}$ located in $\left(0, \frac{\pi}{2}\right)$, we need to shrink to $0 < h\omega_{i} < \sqrt{2}$.

\begin{lem}
\label{lem:coorpro}
With the new coordinate in~(\ref{eqn:coortran}) for $0 < h\omega_{i} < \sqrt{2}$, we have 
\begin{equation}
\label{eqn:thetaphi}
2\phi_{i} + \theta_{i} = \pi
\end{equation}
and
\begin{equation}
\label{eqn:coorpro}
\left\{\begin{aligned}
       & \sin\theta_{i}  = \sin(2\phi_{i}) = h\omega_{i} \sin\phi_{i}\\
       & \sin(3\phi_{i}) = - \left(1 - h^{2}\omega_{i}^{2}\right)\sin\phi_{i}
       \end{aligned}\right. 
\end{equation} 
\end{lem}
\begin{proof}
With Sum-Product identities of trigonometric function, we have 
$$
\begin{aligned}
\sin(\theta_{i} + \phi_{i}) & = \sin\theta_{i} \cos\phi _{i}+ \cos\theta_{i} \sin\phi_{i} \\
                            & = h\omega_{i}\sqrt{1 - \frac{h^{2}\omega_{i}^{2}}{4}} \cdot \frac{h\omega_{i}}{2} + \left(1-\frac{h^{2}\omega_{i}^{2}}{2}\right) \sqrt{1 - \frac{h^{2}\omega_{i}^{2}}{4}}\\
                            & = \sqrt{1 - \frac{h^{2}\omega_{i}^{2}}{4}}\\
                            & = \sin\phi_{i}.
\end{aligned}
$$
Since $0< h\omega_{i} < 2$, we have $\theta_{i}, \phi_{i} \in \left(0, \frac{\pi}{2}\right)$, we can obtain that 
$$
\theta_{i} + \phi_{i} = \pi - \phi_{i} \Leftrightarrow \theta_{i} = \pi - 2\phi_{i}
$$
and with the coordinate transfornation in~(\ref{eqn:coortran}), we have
$$
\sin\theta_{i} = h\omega_{i} \sin\phi_{i} \Leftrightarrow \sin(2\phi_{i}) = h\omega_{i} \sin\phi_{i}.
$$
Next, we use Sum-Product identities of trigonometric function furthermore, 
$$
\begin{aligned}
\sin(\theta_{i} - \phi_{i}) & = \sin\theta_{i} \cos\phi _{i}- \cos\theta_{i} \sin\phi_{i} \\
                            & = h\omega_{i}\sqrt{1 - \frac{h^{2}\omega_{i}^{2}}{4}} \cdot \frac{h\omega_{i}}{2} - \left(1-\frac{h^{2}\omega_{i}^{2}}{2}\right) \sqrt{1 - \frac{h^{2}\omega_{i}^{2}}{4}}\\
                            & = \left(h^{2}\omega_{i}^{2}-1\right)\sqrt{1 - \frac{h^{2}\omega_{i}^{2}}{4}}\\
                            & = -\left(1-h^{2}\omega_{i}^{2}\right) \sin\phi_{i}
\end{aligned}
$$
and with $\theta_{i} = \pi - 2\phi_{i}$, we have
$$
\sin(3\phi_{i}) = -\left(1-h^{2}\omega_{i}^{2}\right) \sin\phi_{i}
$$
\end{proof}
\begin{lem}
\label{lem:equiv_express}
With the new coordinate in~(\ref{eqn:coortran}), the matrix $T_{i}$ ($i = 1, \ldots, n$) in~(\ref{eqn:2by2}) can expressed as below, 
\begin{equation}
\label{eqn:equiv_T}
T_{i} = \frac{1}{\omega_{i} \left(e^{-i\phi_{i}} - e^{i\phi_{i}}\right)}
           \begin{pmatrix}
           1 & 1 \\
           \omega_{i} e^{i\phi_{i}} & \omega_{i} e^{-i\phi_{i}}
           \end{pmatrix}\begin{pmatrix}
           e^{i\theta_{i}} & 0 \\
           0 & e^{-i\theta_{i}}
           \end{pmatrix}\begin{pmatrix}
           \omega_{i} e^{-i\phi_{i}} & -1 \\
           -\omega_{i} e^{i\phi_{i}} & 1
           \end{pmatrix}
\end{equation}
\end{lem}
\begin{proof}
For the coordinate transformation in~(\ref{eqn:coortran}), we have
$$
T_{i} \begin{pmatrix}
       1 \\
       \omega_{i}e^{i\phi_{i}}
      \end{pmatrix} = \begin{pmatrix}
       1 \\
       \omega_{i}e^{i\phi_{i}}
      \end{pmatrix}  e^{i\theta_{i}}     \qquad \text{and}\qquad T_{i} \begin{pmatrix}
       1 \\
       \omega_{i}e^{-i\phi_{i}}
      \end{pmatrix} = \begin{pmatrix}
       1 \\
       \omega_{i}e^{-i\phi_{i}} 
      \end{pmatrix} e^{-i\theta_{i}}
$$
Hence,~(\ref{eqn:equiv_T}) is proved.
\end{proof}

\subsection{The Asymptotic Analysis}
\begin{thm}
\label{thm:recur}
Let the initial value $y_{0}$ and $v_{0}$, after the first $k$ steps without reseting the velocity, the iterative solution~(\ref{eqn:equiv_linear}) with the equivalent form~(\ref{eqn:equivtran}) has the form as below
\begin{equation}
\label{eqn:sol}
 \begin{pmatrix} (U_{1}^{T}y)_{k,i} \\ (U_{1}^{T}v)_{k,i} \end{pmatrix} =T_{i}^{k}  \begin{pmatrix} (U_{1}^{T}y)_{0,i} \\ (U_{1}^{T}v)_{0,i} \end{pmatrix} = \begin{pmatrix}
                    -\frac{\sin(k\theta_{i} - \phi_{i})}{\sin\phi_{i}} & \frac{\sin (k\theta_{i})}{\omega_{i}\sin\phi_{i}} \\
                    -\frac{\omega_{i} \sin (k\theta_{i})}{\sin\phi_{i}} & \frac{\sin(k\theta_{i} + \phi_{i})}{\sin\phi_{i}}
                    \end{pmatrix} \begin{pmatrix} (U_{1}^{T}y)_{0,i} \\ (U_{1}^{T}v)_{0,i} \end{pmatrix}
\end{equation}
\end{thm}
\begin{proof}
With Lemma~\ref{lem:equiv_express} and the coordinate transformation~(\ref{eqn:coortran}), we have
$$
\begin{aligned}
T_{i}^{k} & =\frac{1}{\omega_{i} \left(e^{-i\phi_{i}} - e^{i\phi_{i}}\right)}
           \begin{pmatrix}
           1 & 1 \\
           \omega_{i} e^{i\phi_{i}} & \omega_{i} e^{-i\phi_{i}}
           \end{pmatrix}
           \begin{pmatrix}
           e^{i\theta_{i}} & 0 \\
           0 & e^{-i\theta_{i}}
           \end{pmatrix}^{k} \begin{pmatrix}
           \omega_{i} e^{-i\phi_{i}} & -1 \\
           -\omega_{i} e^{i\phi_{i}} & 1
           \end{pmatrix}\\
                & =\frac{1}{\omega_{i} \left(e^{-i\phi_{i}} - e^{i\phi_{i}}\right)}
           \begin{pmatrix}
           1 & 1 \\
           \omega_{i} e^{i\phi_{i}} & \omega_{i} e^{-i\phi_{i}}
           \end{pmatrix} \begin{pmatrix}
           \omega e^{i(k\theta_{i} -\phi_{i})} & -e^{ik\theta_{i}} \\
           -\omega e^{-i(k\theta_{i} -\phi_{i})} & e^{-ik\theta_{i}}
           \end{pmatrix} \\
             & = \begin{pmatrix}
                    -\frac{\sin(k\theta_{i} - \phi_{i})}{\sin\phi_{i}} & \frac{\sin (k\theta_{i})}{\omega_{i}\sin\phi_{i}} \\
                    -\frac{\omega_{i} \sin (k\theta_{i})}{\sin\phi_{i}} & \frac{\sin(k\theta_{i} + \phi_{i})}{\sin\phi_{i}}
                    \end{pmatrix}  
\end{aligned}
$$
The proof is complete.
\end{proof}
Comparing~(\ref{eqn:sol}) and~(\ref{eqn:equivtran}), we can obtain that 
$$
\frac{\sin(k\theta_{i} - \phi_{i})}{\sin\phi_{i}} = 1 - h^{2}\omega_{i}^{2}.
$$
With the intial value $(y_{0},0)^{T}$, then the initial value for~(\ref{eqn:equivtran}) is $(U_{1}^{T}y_{0}, 0)$. In order to make sure the numerical solution, or the iterative solution owns the same behavior as the analytical solution, we need to set $0< h\omega_{i} < 1$. 
\begin{rem}
Here, the behavior is similar as the thought in~\citep{lee2016gradient}. The step size $0<hL<2$ make sure the global convergence of gradient method. And the step size $0<hL<1$ make the uniqueness of the trajectory along the gradient method, the thought of which is equivalent of the existencen and uniqueness of the solution for ODE. Actually, the step size $0<hL<1$ owns the property with the solution of ODE, the continous-limit version. A global existence of the solution for gradient system is proved in~\citep{perko2013differential}.
\end{rem}
For the good-conditioned eigenvalue of the Hessian $\nabla^{2} f(x^{\star})$, every method such as gradient method, momentum method, Nesterov accelerated gradient method and artificially dissipating energy method has the good convergence rate shown by the experiment. However, for our artificially dissipating energy method, since there are trigonometric functions from~(\ref{eqn:sol}), we cannot propose the rigorous mathematic proof for the convergence rate. If everybody can propose a theoretical proof, it is very beautiful. Here, we propose a theoretical approximation for ill-conditioned case, that is, the direction with small eigenvalue $\lambda (\nabla^{2}f(x^{\star})) \ll L$. 
\begin{assumption}
\label{ass:ill}
If the step size $h = \frac{1}{\sqrt{L}}$ for~(\ref{eqn:equiv_linear}), for the ill-conditioned eigenvalue $\omega_{i} \ll \sqrt{L}$, the coordinate variable can be approximated by the analytical solution as 
\begin{equation}
\label{eqn:approx}
\theta_{i} = h\omega_{i}, \qquad \text{and}\qquad\phi_{i} = \frac{\pi}{2}.
\end{equation}
\end{assumption}
With Assumption~\ref{ass:ill}, the iterative solution~(\ref{eqn:sol}) can be rewritten as 
\begin{equation}
\label{eqn:ass_vartran}
 \begin{pmatrix} (U_{1}^{T}y)_{k,i} \\ (U_{1}^{T}v)_{k,i} \end{pmatrix} = \begin{pmatrix} \cos(kh\omega_{i}) & \frac{\sin(kh\omega_{i})}{\omega_{i}}\\ -\omega_{i}\sin(kh\omega_{i}) & -\cos(kh\omega_{i}) \end{pmatrix}\begin{pmatrix} (U_{1}^{T}y)_{0,i} \\ (U_{1}^{T}v)_{0,i} \end{pmatrix}
\end{equation}

\begin{thm}
For every ill-conditioned eigen-direction, with every initial condition $(y_{0},0)^{T}$, if the algorithm~\ref{alg:local} is implemented at $\left\|v_{iter}\right\| \leq \left\|v\right\|$, then there exist an eigenvalue $\omega_{i}^{2}$ such that 
$$k\omega_{i} h \geq \frac{\pi}{2}.$$
\end{thm}
\begin{proof}
When $\left\|v_{iter}\right\| \leq \left\|v\right\|$, then $\left\|U_{1}^{T}v_{iter}\right\| \leq \left\|U_{1}^{T}v\right\|$. While for the $\left\|U_{1}^{T}v\right\|$, we can write in the analytical form, 
$$
\left\|U_{1}^{T}v\right\| = \sqrt{\sum_{i=1}^{n} \omega_{i}^{2}(U_{1}y_{0})_{i}^{2}\sin^{2}(kh\omega_{i})}
$$
if there is no $k\omega_{i} h < \frac{\pi}{2}$, $\left\|U_{1}^{T}v\right\|$ increase with $k$ increasing.
\end{proof}

For some $i$ such that $k\omega_{i}h$ approximating $\frac{\pi}{2}$, we have
\begin{equation}
\label{eqn:decayrapid}
\begin{aligned}
\frac{\left|(U_{1}^{T}y)_{k+1,i}\right|}{\left|(U_{1}^{T}y)_{k,i}\right|} & = \frac{\cos\left((k+1)h\omega_{i}\right)}{\cos\left(kh\omega_{i}\right)} \\
                                                                                                        & = e^{\ln\cos\left((k+1)h\omega_{i}\right) - \ln \cos\left(kh\omega_{i}\right) } \\
                                                                                                        & = e^{-\tan(\xi) h\omega_{i}}
\end{aligned}
\end{equation}
where $\xi \in \left(kh\omega_{i}, (k+1)h\omega_{i}\right)$. Hence, with $\xi$ approximating $\frac{\pi}{2}$, $\left|(U_{1}^{T}y)_{k,i}\right|$ approximatie $0$ with the linear convergence, but the coefficient will also decay with the rate $e^{-\tan(\xi) h\omega_{i}}$ with $\xi \rightarrow \frac{\pi}{2}$. With the Laurent expansion for $\tan \xi$ at $\frac{\pi}{2}$, i.e.,
$$
\tan \xi = -\frac{1}{\xi - \frac{\pi}{2}} +\frac{1}{3} \left(\xi - \frac{\pi}{2}\right) + \frac{1}{45} \left(\xi - \frac{\pi}{2}\right)^{3} + \mathcal{O}\left(\left(\xi - \frac{\pi}{2}\right)^{5}\right)
$$  
the coefficient has the approximating formula
$$ 
e^{-\tan(\xi) h\omega_{i}} \approx e^{ \frac{h\omega_{i}}{\xi - \frac{\pi}{2}}} \leq \left(\frac{\pi}{2} -\xi \right)^{n}.
$$
where $n$ is an arbitrary large real number in $\mathbb{R}^{+}$ for $\xi \rightarrow \frac{\pi}{2}$.

%% file: 05_experimental.tex
\section{Experimental Demonstration}
\label{sec:experimental}
In this section, we implement the artificially dissipating energy algorithm (algorithm~\ref{alg:local}), energy conservation algorithm (algorithm~\ref{alg:global}) and the combined algorithm (algorithm~\ref{alg:combined}) into  high-dimension data for comparison with gradient method, momentum method and Nesterov accelerated gradient method.
\subsection{Strongly Convex Function}
Here, we investigate the artificially dissipating energy algorithm (algorithm~\ref{alg:local}) for the strongly convex function for comparison with gradient method, momentum method and Nesterov accelerated gradient method (strongly convex case) by the quadratic function as below. 
\begin{equation}
\label{eqn:strong}
f(x) = \frac{1}{2}x^{T} Ax + b^{T}x
\end{equation}
where $A$ is symmetric and positive-definite matrix.  The two cases are shown as below:
\begin{enumerate}[label=\textbf{(\alph*)}]
\item The generative matrix $A$ is $500\times 500$ random positive definite matrix with eigenvalue from $1e-6$ to $1$ with one defined eigenvalue $1e-6$. The generative vector $b$ follows  i.i.d. Gaussian distribution with mean $0$ and variance $1$.  
\item The generative matrix $A$ is the notorious example in Nesterov's book~\citep{nesterov2013introductory}, i.e.,
         $$
         A = \begin{pmatrix}
                2 & -1 &   &  &&\\
                -1 &2  &-1&  &&\\
                    &-1 & 2 &\ddots  &&\\
                    &     &\ddots  &\ddots &\ddots &            \\
                    &     &  &\ddots  &\ddots       &-1   \\
                    &      &          &   & -1   &2  
               \end{pmatrix}
         $$
         the eigenvalues of the matrix are
         $$
         \lambda_{k} = 2 - 2\cos\left(\frac{k\pi}{n+1}\right) = 4\sin^{2}\left(\frac{k\pi}{2(n+1)}\right)
         $$ 
         and $n$ is the dimension of the matrix $A$. The eigenvector can be solved by the second Chebyshev's polynomial. We implement $\dim(A) = 1000$   
         and $b$ is zero vector.  Hence, the smallest eigenvalue is approximating
         $$
         \lambda_{1} = 4\sin^{2}\left(\frac{\pi}{2(n+1)}\right) \approx \frac{\pi^{2}}{1001^{2}} \approx 10^{-5}
         $$
\end{enumerate}
\begin{figure}[H]
\begin{minipage}[t]{0.5\linewidth}
\centering
\includegraphics[width=3.2in]{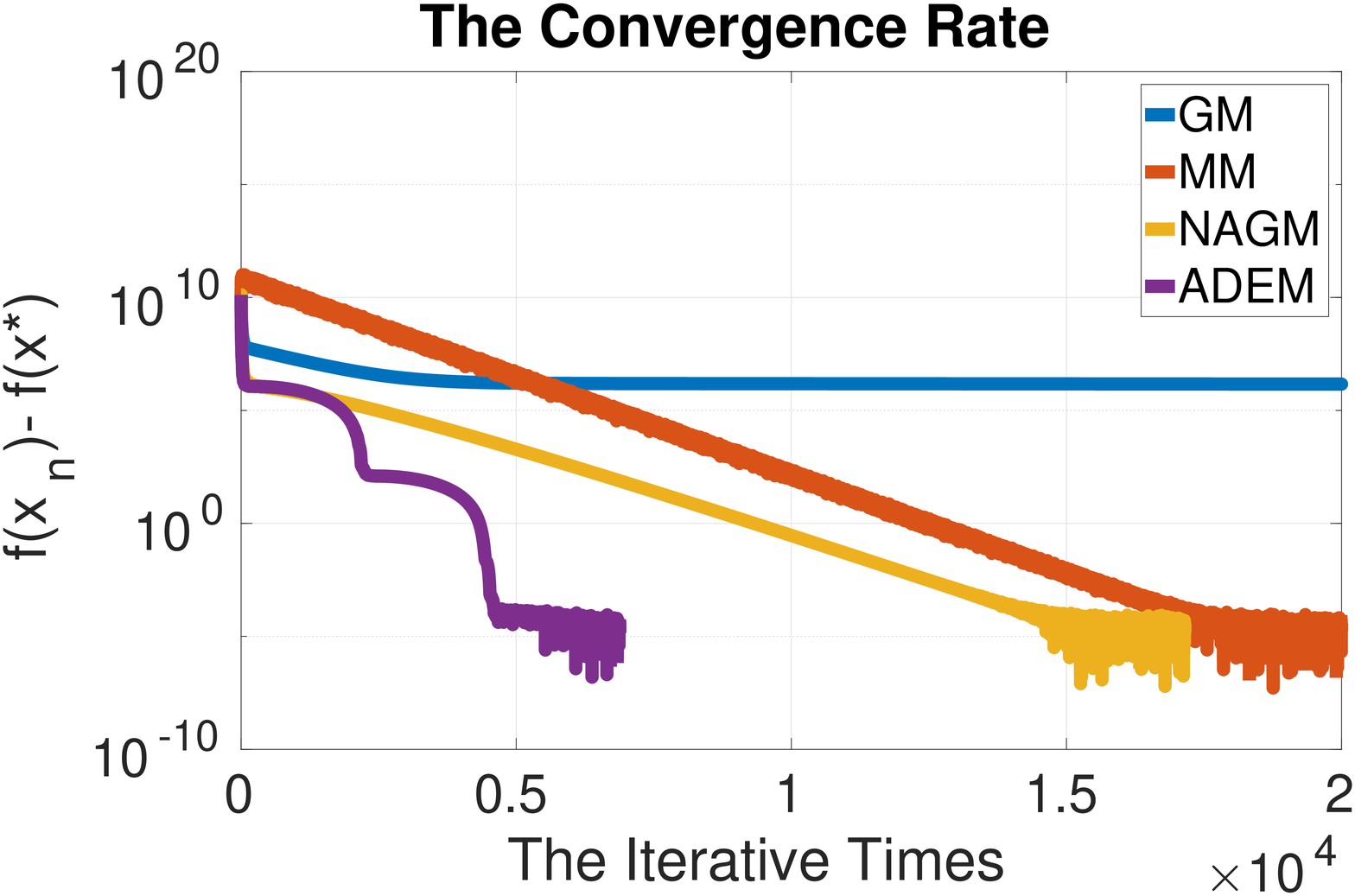}
\end{minipage}%
\begin{minipage}[t]{0.5\linewidth}
\centering
\includegraphics[width=3.2in]{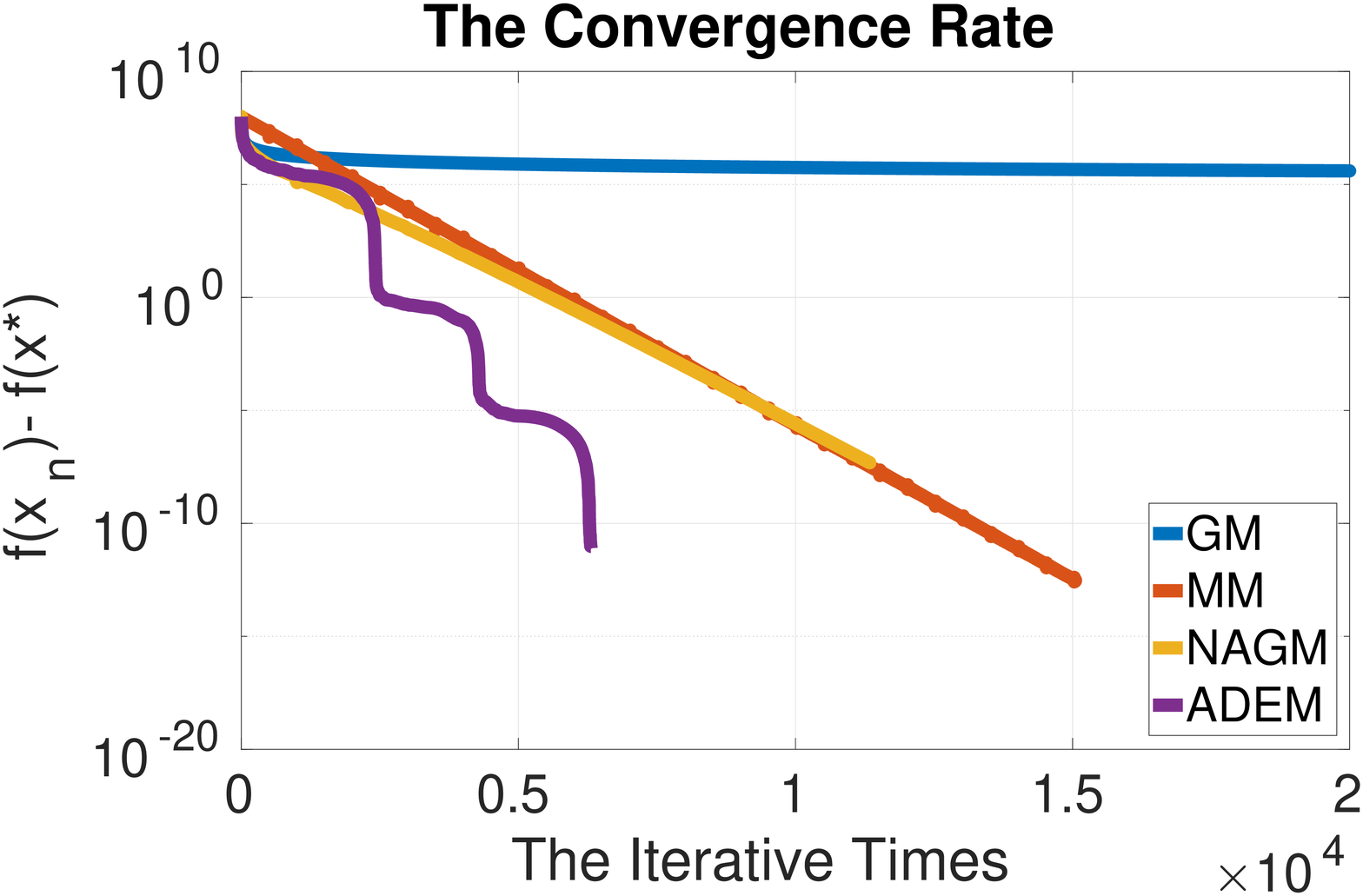}
\end{minipage}
\caption{The Left: the case ($\mathbf{a}$) with the initial point $x_{0} = 0$. The Right: the case ($\mathbf{b}$) with the initial point $x_{0} = 1000$} 
\label{fig:strongly}
\end{figure}

\subsection{Non-Strongly Convex Function}
Here, we investigate the artificially dissipating energy algorithm (algorithm~\ref{alg:local}) for the non-strongly convex function for comparison with gradient method, Nesterov accelerated gradient method (non-strongly convex case) by the log-sum-exp function as below. 
\begin{equation}
\label{eqn:nonstrong}
f(x) = \rho \log \left[\sum_{i=1}^{n}\exp\left( \frac{\langle a_{i}, x\rangle - b_{i}}{\rho}\right)\right]
\end{equation}
where $A$ is the $m\times n$ matrix with $a_{i}$, $(i=1,\ldots,m)$ the column vector of $A$ and $b$ is the $n\times 1$ vector with component $b_{i}$. $\rho$ is the parameter. We show the experiment in~(\ref{eqn:nonstrong}): the matrix $A = \left(a_{ij}\right)_{50\times 200}$ and the vector $b = (b_{i})_{200\times 1}$  are set by the entry following i.i.d Gaussian distribution for the paramter $\rho=5$ and $\rho=10$. 
\begin{figure}[H]
\begin{minipage}[t]{0.5\linewidth}
\centering
\includegraphics[width=3.2in]{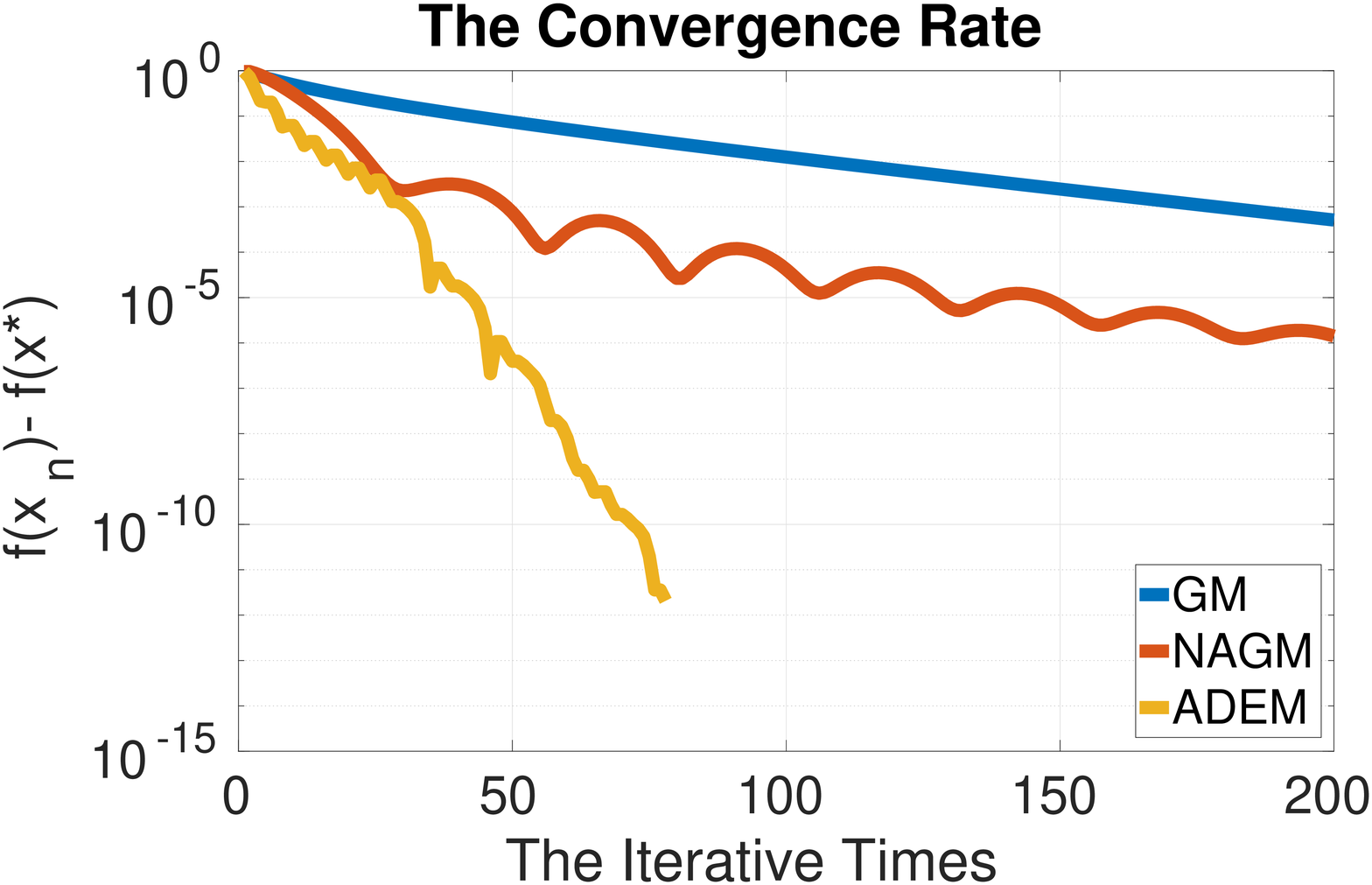}
\end{minipage}%
\begin{minipage}[t]{0.5\linewidth}
\centering
\includegraphics[width=3.2in]{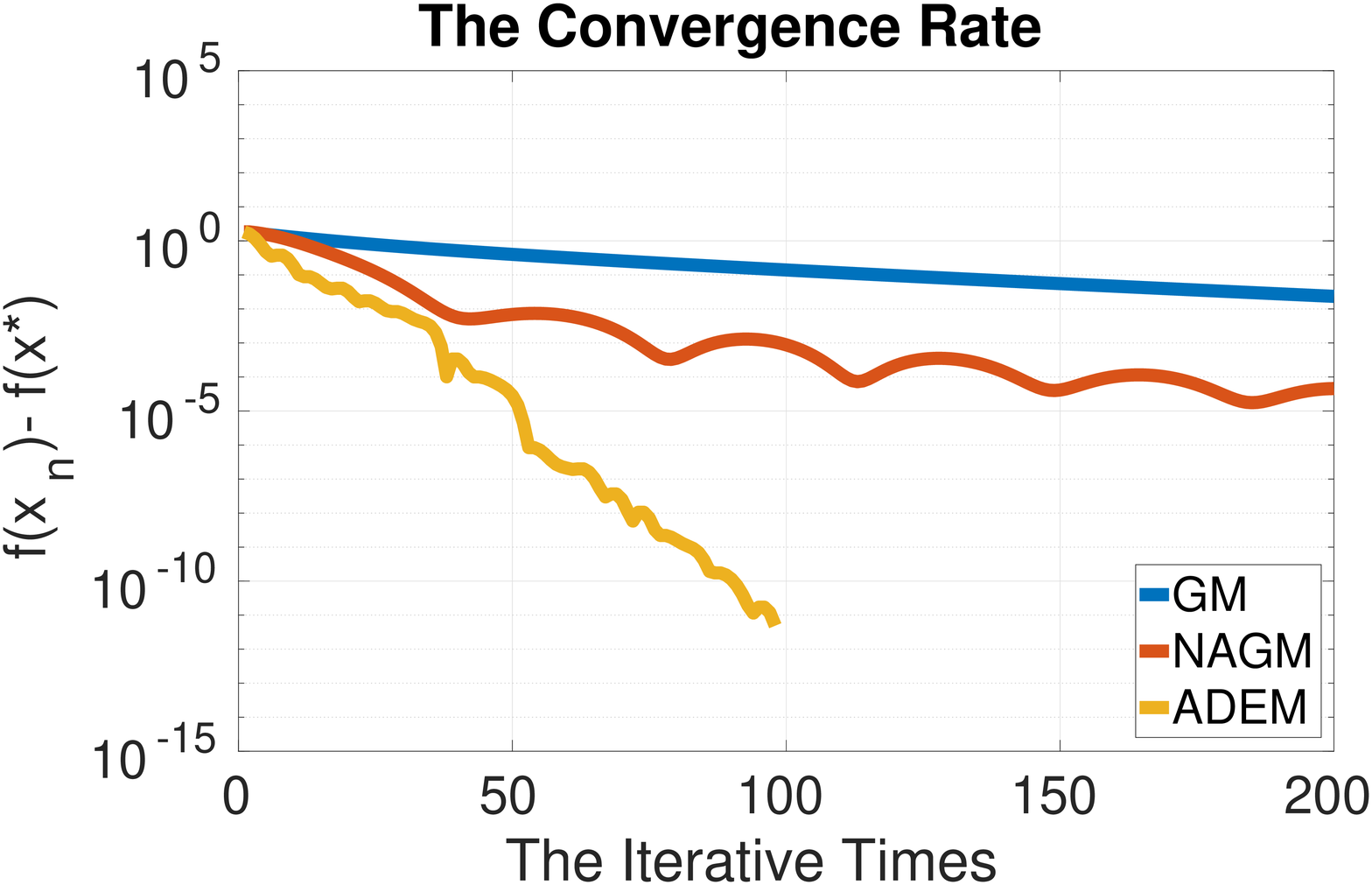}
\end{minipage}
\caption{The convergence rate is shown from the initial point $x_{0}=0$. The Left: $\rho =5$; The Right: $\rho = 10$.} 
\label{fig:nonconvex}
\end{figure}

\subsection{Non-convex Function}
For the nonconvex function, we exploit classical test function, known as artificial landscape, to evaluate characteristics of optimization algorithms from general performance and precision. In this paper, we show our algorithms implementing on the Styblinski-Tang function and Shekel function, which is recorded in the virtual library of simulation experiments\footnote{https://www.sfu.ca/~ssurjano/index.html}. Firstly, we investigate Styblinski-Tang function, i.e. 
\begin{equation}
\label{eqn:styblinski_tang}
f(x) = \frac{1}{2} \sum_{i=1}^{d} \left(x_{i}^{4} - 16x_{i}^{2} + 5x_{i}\right)
\end{equation}
to demonstrate the general performance of the algorithm~\ref{alg:global} to track the number of local minima and then find the local minima by algorithm~\ref{alg:combined}. 
\begin{figure}[H]
\begin{minipage}[t]{0.5\linewidth}
\centering
\includegraphics[width=3.2in]{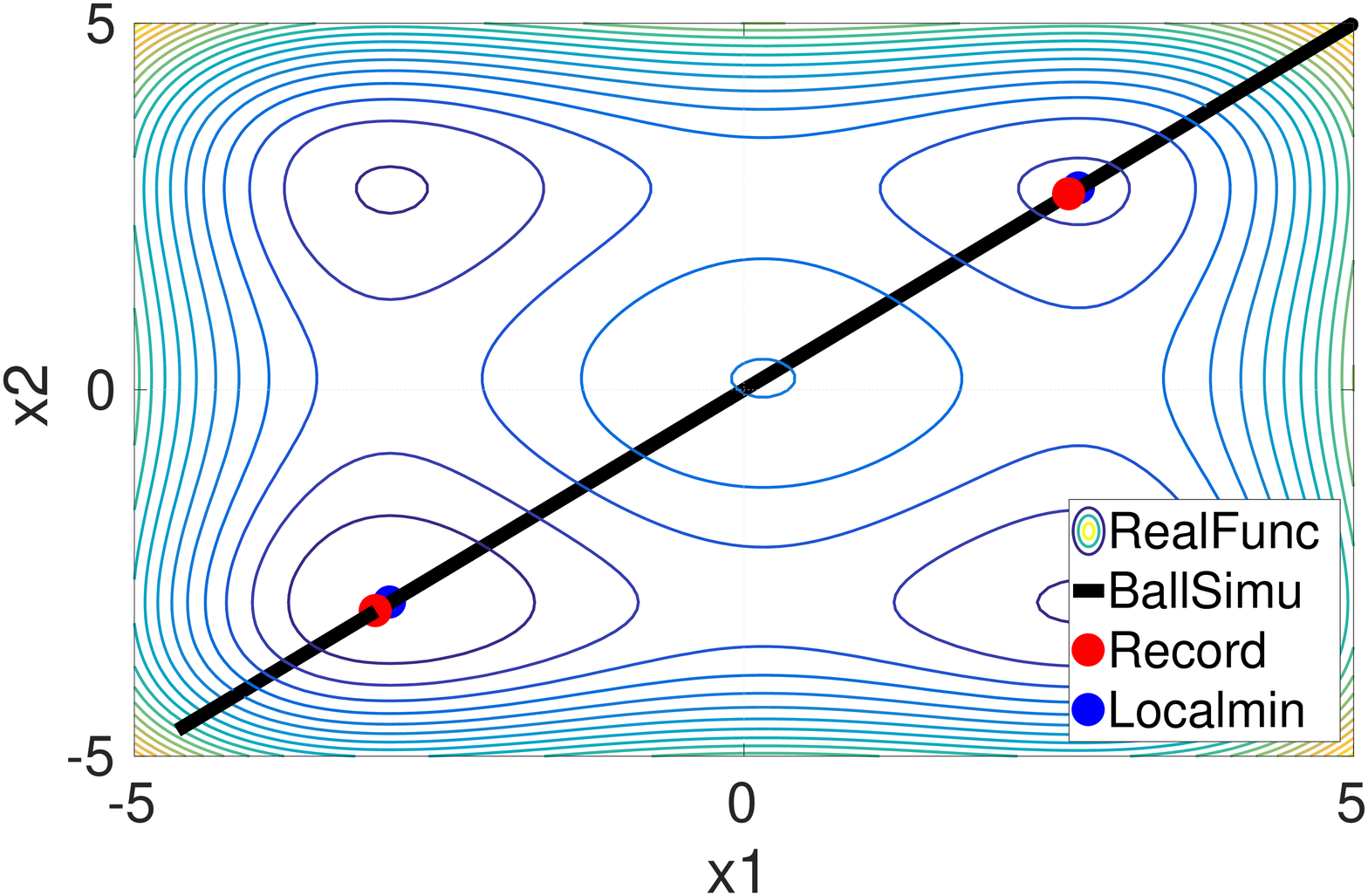}
\end{minipage}%
\begin{minipage}[t]{0.5\linewidth}
\centering
\includegraphics[width=3.2in]{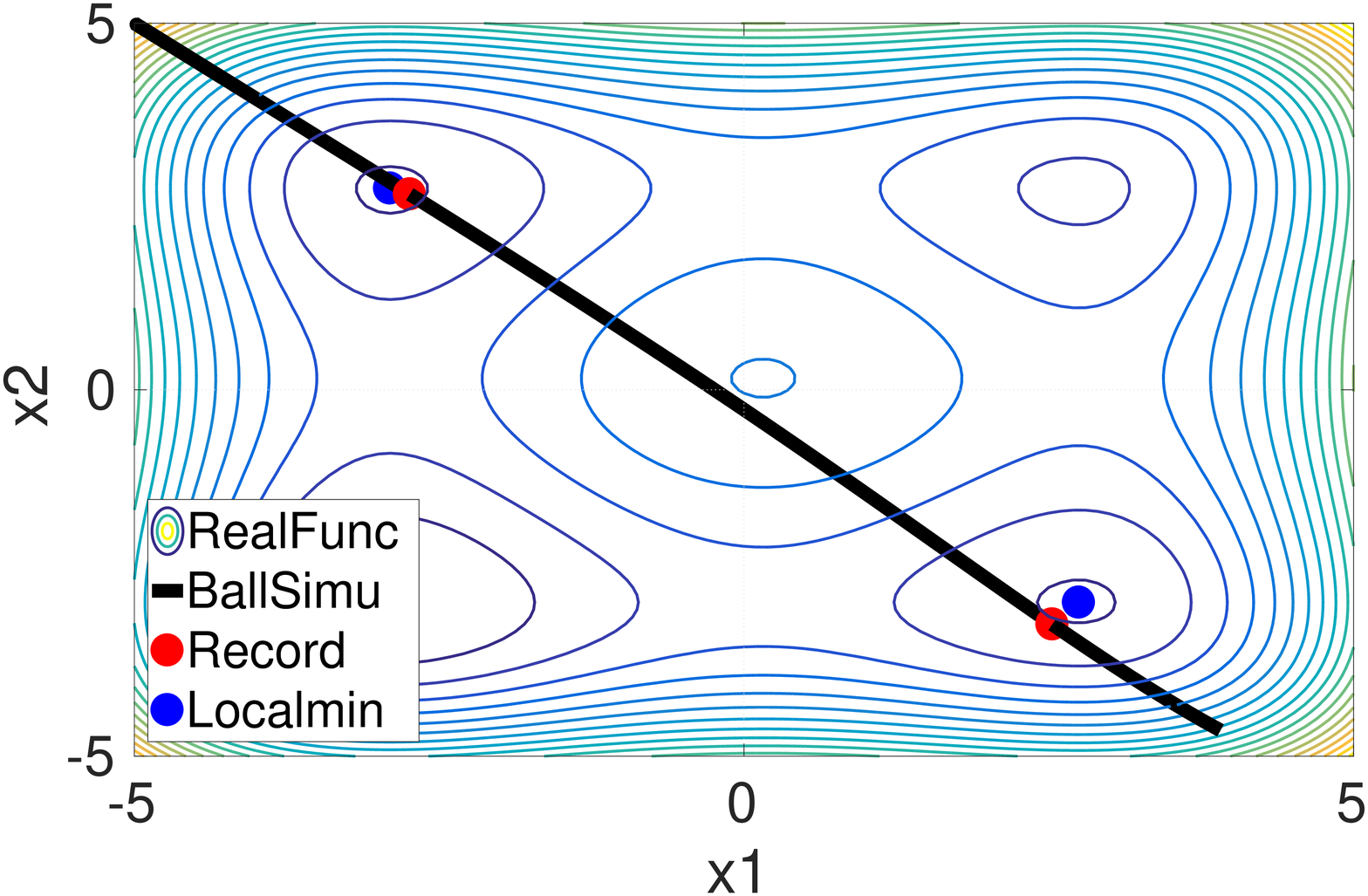}
\end{minipage}
\caption{Detecting the number of the local minima of 2-D Styblinski-Tang function by algorithm~\ref{alg:combined} with step length $h = 0.01$. The red points are recorded by algorithm~\ref{alg:global} and the blue point are the local minima by algorithm~\ref{alg:local}. The Left: The Initial Position $(5,5)$; The Right: The Initial Position $(-5,5)$.} 
\label{fig:styblinskitang}
\end{figure}

To the essential 1-D nonconvex Styblinski-Tang function of high dimension, we implement the algorithm~\ref{alg:combined} to obtain the precision of the global minima as below.
\begin{table}[H]
\begin{tabular}{lllll}
  \toprule
      &Local\_min1  &Local\_min2 &Local\_min3 &Local\_min4 \\
  \midrule
Initial Position      &(5,5,\ldots)                          & (5,5,\ldots)             & (5,-5,\ldots)          &(5,-5,\ldots) \\
       Position       &(2.7486,2.7486,\ldots)                & (-2.9035,-2.9035,\ldots) & (2.7486,-2.9035,\ldots)&(-2.9035,2.7486,\ldots) \\
Function Value        &-250.2945      &-391.6617 &-320.9781 &-320.9781 \\
  \bottomrule
\end{tabular}
\caption{The example for ten-dimensional Styblinski-Tang function from two initial positions.}
\label{label:1}
\end{table}

The global minima calculated at the position $(-2.9035,-2.9035,\ldots)$ is $-391.6617$ shown on the Table~\ref{label:1}. And the real global minima  at $(-2.903534,-2.903534,\ldots)$ is $-39.16599\times 10 = -391.6599$.

Furthermore, we demonstrate the numerical experiment from Styblinski-Tang function to more complex Shekel function 
\begin{equation}
\label{eqn:shekel}
f(x) = - \sum_{i=1}^{m} \left(\sum_{j=1}^{4}\left(x_{j} - C_{ji}\right)^{2} + \beta_{i}\right)^{-1}
\end{equation} 
where $$\beta = \frac{1}{10}\left(1, 2, 2, 4, 4, 6, 3, 7, 5, 5\right)^{T}$$
and
$$
C = \begin{pmatrix}
    4.0 & 1.0 & 8.0 & 6.0 & 3.0 & 2.0 & 5.0 & 8.0 & 6.0 & 7.0 \\
    4.0 & 1.0 & 8.0 & 6.0 & 7.0 & 9.0 & 3.0 & 1.0 & 2.0 & 3.6 \\
    4.0 & 1.0 & 8.0 & 6.0 & 3.0 & 2.0 & 5.0 & 8.0 & 6.0 & 7.0 \\
    4.0 & 1.0 & 8.0 & 6.0 & 7.0 & 9.0 & 3.0 & 1.0 & 2.0 & 3.6 
    \end{pmatrix}.
$$
\begin{enumerate}[label=\textbf{(\arabic*)}]
\item Case $m=5$, the global minima at $x^{\star} = (4,4,4,4)$ is $f(x^{\star}) = - 10.1532$.
\begin{enumerate}[label=\textbf{(\alph*)}]
\item From the position $(10,10,10,10)$, the experimental result with the step length $h=0.01$ and the iterative times $3000$ is shown as below

Detect Position (Algorithm~\ref{alg:global})
$$
\begin{pmatrix}
7.9879 & 6.0136 & 3.8525 & 6.2914 & 2.7818 \\
7.9958 & 5.9553 & 3.9196 & 6.2432 & 6.7434 \\
7.9879 & 6.0136 & 3.8525 & 6.2914 & 2.7818 \\
7.9958 & 5.9553 & 3.9196 & 6.2432 & 6.7434
\end{pmatrix}
$$
Detect value 
$$
\begin{pmatrix}
-5.0932 & -2.6551 & -6.5387 & -1.6356 & -1.7262
\end{pmatrix}
$$
Final position (Algorithm~\ref{alg:local}) 
$$
\begin{pmatrix}
7.9996 & 5.9987 & 4.0000 & 5.9987 & 3.0018 \\
7.9996 & 6.0003 & 4.0001 & 6.0003 & 6.9983 \\
7.9996 & 5.9987 & 4.0000 & 5.9987 & 3.0018 \\
7.9996 & 6.0003 & 4.0001 & 6.0003 & 6.9983
\end{pmatrix}
$$
Final value
$$
\begin{pmatrix}
-5.1008 & -2.6829 & -10.1532 & -2.6829 & -2.6305
\end{pmatrix}
$$
\item From the position $(3,3,3,3)$, the experimental result with the step length $h=0.01$ and the iterative times $1000$ is shown as below

Detect Position (Algorithm~\ref{alg:global})
$$
\begin{pmatrix}
3.9957 & 6.0140  \\
4.0052 & 6.0068  \\
3.9957 & 6.0140  \\
4.0052 & 6.0068 
\end{pmatrix}
$$
Detect value
$$
\begin{pmatrix}
-10.1443 & -2.6794 
\end{pmatrix}
$$
Final position (Algorithm~\ref{alg:local})
$$
\begin{pmatrix}
4.0000 & 5.9987  \\
4.0001 & 6.0003  \\
4.0000 & 5.9987 \\
4.0001 & 6.0003
\end{pmatrix}
$$
Final value
$$
\begin{pmatrix}
-10.1532 & -2.6829 
\end{pmatrix}
$$
\end{enumerate}
\item Case $m=7$, the global minima at $x^{\star} = (4,4,4,4)$ is $f(x^{\star}) = - 10.4029$.
\begin{enumerate}[label=\textbf{(\alph*)}]
\item From the position $(10,10,10,10)$, the experimental result with the step length $h=0.01$ and the iterative times $3000$ is shown as below

Detect Position (Algorithm~\ref{alg:global})
$$
\begin{pmatrix}
7.9879 & 6.0372 & 3.1798 & 5.0430 & 6.2216 & 2.6956 \\
8.0041 & 5.9065 & 3.8330 & 2.8743 & 6.2453 & 6.6837 \\
7.9879 & 6.0372 & 3.1798 & 5.0430 & 6.2216 & 2.6956 \\
8.0041 & 5.9065 & 3.8330 & 2.8743 & 6.2453 & 6.6837
\end{pmatrix}
$$
Detect value 
$$
\begin{pmatrix}
-5.1211 & -2.6312 & -0.9428 & -3.3093 & -1.8597 & -1.5108
\end{pmatrix}
$$
Final position  (Algorithm~\ref{alg:local})
$$
\begin{pmatrix}
7.9995 & 5.9981 & 4.0006 & 4.9945 & 5.9981 & 3.0006 \\
7.9996 & 5.9993 & 3.9996 & 3.0064 & 5.9993 & 7.0008 \\
7.9995 & 5.9981 & 4.0006 & 4.9945 & 5.9981 & 3.0006 \\
7.9996 & 5.9993 & 3.9996 & 3.0064 & 5.9993 & 7.0008
\end{pmatrix}
$$
Final value
$$
\begin{pmatrix}
-5.1288 & -2.7519 & -10.4029 & -3.7031 & -2.7519 & -2.7496
\end{pmatrix}
$$
\item From the position $(3,3,3,3)$, the experimental result with the step length $h=0.01$ and the iterative times $1000$ is shown as below

Detect Position (Algorithm~\ref{alg:global})
$$
\begin{pmatrix}
4.0593 & 3.0228  \\
3.9976 & 7.1782  \\
4.0593 & 3.0228  \\
3.9976 & 7.1782 
\end{pmatrix}
$$
Detect value
$$
\begin{pmatrix}
-9.7595 & -2.4073 
\end{pmatrix}
$$
Final position (Algorithm~\ref{alg:local})
$$
\begin{pmatrix}
4.0006 & 3.0006  \\
3.9996 & 7.0008  \\
4.0006 & 3.0006 \\
3.9996 & 7.0008
\end{pmatrix}
$$
Final value
$$
\begin{pmatrix}
-10.4029 & -2.7496
\end{pmatrix}
$$
\end{enumerate}
\item Case $m=10$, the global minima at $x^{\star} = (4,4,4,4)$ is $f(x^{\star}) = - 10.5364$.
\begin{enumerate}[label=\textbf{(\alph*)}]
\item From the position $(10,10,10,10)$, the experimental result with the step length $h=0.01$ and the iterative times $3000$ is shown as below

Detect Position (Algorithm~\ref{alg:global})
$$
\begin{pmatrix}
7.9977 & 5.9827 & 4.0225 & 2.7268 & 6.1849 & 6.2831 & 6.3929 \\
7.9942 & 6.0007 & 3.8676 & 7.3588 & 6.0601 & 3.2421 & 1.9394 \\
7.9977 & 5.9827 & 4.0225 & 2.7268 & 6.1849 & 6.2831 & 6.3929 \\
7.9942 & 6.0007 & 3.8676 & 7.3588 & 6.0601 & 3.2421 & 1.9394
\end{pmatrix}
$$
 Detect value 
$$
\begin{pmatrix}
-5.1741 & -2.8676 & -7.9230 & -1.5442 & -2.4650 &-1.3703 &-1.7895
\end{pmatrix}
$$
Final position (Algorithm~\ref{alg:local})
$$
\begin{pmatrix}
7.9995 & 5.9990 & 4.0007 & 3.0009 & 5.9990 & 6.8999 &5.9919 \\
7.9994 & 5.9965 & 3.9995 & 7.0004 & 5.9965 & 3.4916 &2.0224 \\
7.9995 & 5.9990 & 4.0007 & 3.0009 & 5.9990 & 6.8999 &5.9919 \\
7.9994 & 5.9965 & 3.9995 & 7.0004 & 5.9965 & 3.4916 &2.0224
\end{pmatrix}
$$
Final value
$$
\begin{pmatrix}
-5.1756 & -2.8712 & -10.5364 & -2.7903 & -2.8712 & -2.3697 & -2.6085
\end{pmatrix}
$$
\item From the position $(3,3,3,3)$, the experimental result with the step length $h=0.01$ and the iterative times $1000$ is shown as below

Detect Position (Algorithm~\ref{alg:global})
$$
\begin{pmatrix}
4.0812 & 3.0206  \\
3.9794 & 7.0173  \\
4.0812 & 3.0206  \\
3.9794 & 7.0173 
\end{pmatrix}
$$
Detect value
$$
\begin{pmatrix}
-9.3348 & -2.7819
\end{pmatrix}
$$
Final position (Algorithm~\ref{alg:local})
$$
\begin{pmatrix}
4.0007 & 3.0009  \\
3.9995 & 7.0004  \\
4.0007 & 3.0009 \\
3.9995 & 7.0004
\end{pmatrix}
$$
Final value
$$
\begin{pmatrix}
-10.5364 & -2.7903
\end{pmatrix}
$$
\end{enumerate}
\end{enumerate}

%% file: 06_concl.tex
\section{Conclusion and Further Works}
\label{sec:conclusion}

Based on the view for understanding arithmetical complexity from analytical complexity in the seminal book~\citep{nesterov2013introductory} and the idea for viewing optimization from differential equation in the novel blog\footnote{http://www.offconvex.org/2015/12/11/mission-statement/} , we propose  some original algorithms based on Newton Second Law with the kinetic energy observable and controllable in the computational process firstly. Although our algorithm cannot fully solve the global optimization problem, or it is dependent on the trajectory path, this work introduces time-independent Hamilton system essentially to optimization such that it is possible that the global minima can be obtained. Our algorithms are easy to implement and own more rapid convergence rate.

For the theoretical view, the time-independent Hamilton system is  closer to nature and a lot of fundamental work have appeared in the previous century, such as KAM theory, Nekhoroshev estimate, operator spectral theory and so on~\citep{arnol2013mathematical,arnol?d2012geometrical}. Are these beautiful and essentially original work used to understand and improve the algorithm for optimization and machine learning? Also, to estimate the convergence rate,  the matrix containing the trigonometric function is hard to estimate. Some estimate for the trigonometric matrix based on spectral theory are proposed in~\citep{jitomirskaya2017arithmetic,liu2015anderson}. For the numerical scheme, we only exploit the simple first-order symplectic Euler method.  A lot of more efficient schemes, such as St\"{o}rmer-Verlet scheme, Symplectic Runge-Kutta scheme, order condition method and so on, are proposed on~\citep{hairer2006geometric}.  These schemes can make the algorithms in this paper more efficient and accurate.   For the optimization, the method we proposed is only about unconstrained problem. In the nature, the classical Newton Second law, or the equivalent expression --- Lagrange mechanics and Hamilton mechanics, is implemented on the manifold in the almost real physical world. In other word, a natural generalization is from unconstrained problem to constrained problem for our proposed algorithms. A more natural implementation is the geodesic descent in~\citep{luenberger1984linear}. Similar as the development of the gradient method from smooth condition to nonsmooth condition, our algorithms can be generalized to nonsmooth condition by the subgradient. For application, we will implement our algorithms to Non-negative Matrix Factorization, Matrix Completion and Deep Neural Network and speed up the training of the objective function. Meanwhile, we apply the algorithms proposed in this paper to the maximum likelihood estimator and maximum a posteriori estimator in statistics.

Starting from Newton Second Law, we implement only a simple particle in classical mechanics, or macroscopic world. A natural generalization is from the macroscopic world to the microscopic world. In the field of fluid dynamics, the Newton second Law is expressed by Euler equation, or more complex Navier-Stokes equation. An important topic from fluid dynamics is geophysical fluid dynamics~\citep{pedlosky2013geophysical,cushman2011introduction} , containing atmospheric science and oceanography. Especially, a key feature in the oceanography different from atmospheric science is the topography, which influence mainly vector field of the fluid. So many results have been demonstrated based on many numerical modeling , such as the classical POM\footnote{ http://ofs.dmcr.go.th/thailand/model.html}, HYCOM\footnote{https://hycom.org/}, ROMS\footnote{https://www.myroms.org/} and FVCOM\footnote{http://fvcom.smast.umassd.edu/}. A reverse idea is that if we view the potential function in black box is the topography, we observe the changing of the fluid vector field to find the number of local minima in order to obtain the global minima with a suitable initial vector field.  A more adventurous idea is to generalize the classical particle to the quantum particle. For quantum particle, the Newton second law is expressed by the energy form, that is from the view of Hamilton mechanics, which is the starting point for the proposed algorithm in this paper.  The particle appears in wave form in microscopic world. When the wave meets the potential barrier, the tunneling phenomena will appear. The tunneling phenomena still appear in high dimension~\citep{nakamura2013quantum}. It is very easy to observe the tunneling phenomena in the physical world. If the computer can be very easy to simulate the quantum world, we can find the global minima by binary section search. That is, if there exist tunneling phenomena in the upper level, continue to detect the upper level in the upper level, otherwise to go the lower level. In quantum world, it need only $\mathcal{O}(\log n)$ times to find global minima other than NP-hard.